\documentclass[11pt,a4paper,fleqn]{article}
\pdfoutput=1
\RequirePackage{amsthm,amsmath,natbib}
\RequirePackage{amssymb,amstext}
\RequirePackage{hypernat}
\usepackage[pdfpagemode=UseOutlines ,plainpages=false
,hypertexnames=false ,pdfpagelabels ,hyperindex=true,colorlinks=true]{hyperref}
	\makeatletter%
	\Hy@breaklinkstrue%
	\makeatother%
\usepackage{color}
\definecolor{darkred}{rgb}{0.6,0.0,0.1}
\definecolor{darkgreen}{rgb}{0,0.5,0}
\definecolor{darkblue}{rgb}{0,0,0.5}
\hypersetup{colorlinks ,linkcolor=darkblue ,filecolor=darkgreen
,urlcolor=darkblue ,citecolor=black}

\usepackage{		url}
\usepackage{verbatim}
\usepackage{ bm}

\renewcommand{\cite}{\citet}
\usepackage{jan-christoph-abrev-package}

\definecolor{dgreen}{rgb}{0,0.5,0}
\definecolor{dblue}{rgb}{0,0,0.9}
\definecolor{dred}{rgb}{0.6,0.0,0.1}
\definecolor{dgold}{rgb}{0.5,0.3,0.0}
\definecolor{dvio}{rgb}{0.6,0.3,0.5}
\definecolor{gray}{rgb}{0.5,0.5,0.5}

\newcommand{\dr}{\color{dred}}

\newcommand{\dgr}{\color{gray}}

\newtheoremstyle{mysc}
  {3pt}
  {3pt}
  {\it}
  {}
  {\color{darkred}\sc}
  {.}
  {.5em}
  {}

\newtheoremstyle{myex}
  {10pt}
  {10pt}
  {\rm}
  {}
  {\color{darkred}\sc}
  {.}
  {.5em}
  {}

\theoremstyle{mysc}\newtheorem{prop}{Proposition}[section]
\theoremstyle{mysc}\newtheorem{assumption}{Assumption}[section]
\theoremstyle{mysc}\newtheorem{coro}[prop]{Corollary}
\theoremstyle{mysc}\newtheorem{theo}[prop]{Theorem}
\theoremstyle{mysc}
\theoremstyle{mysc}\newtheorem{lem}[prop]{Lemma}
\theoremstyle{myex}\newtheorem{rem}{Remark}[section]
\theoremstyle{myex}
\theoremstyle{myex}

\theoremstyle{mysc}\newtheorem{assA}{Assumption}

\theoremstyle{mysc}
\theoremstyle{mysc}
\theoremstyle{mysc}

\numberwithin{equation}{section}

\newcommand{\note}[1]
{$^{(!)}$\marginpar[{\hfill\tiny{\sf{#1}}}]{\tiny{\sf{(!) #1}}}}

\author{{\sc Christoph Breunig}$^*$ and {\sc Jan Johannes}\thanks{Universit\"at Heidelberg, Institut f\"ur Angewandte Mathematik, Im Neuenheimer Feld, 294, D-69120 Heidelberg, Germany. Corresponding author Email: \url{johannes@statlab.uni-heidelberg.de}.}}

\title{{\bf On rate optimal local estimation\\ in nonparametric instrumental regression.}}

\begin{document}
\date{\today}
\maketitle


\begin{abstract} We consider the problem of estimating the value of a linear functional in  nonparametric instrumental regression, where  in the presence of an instrument $W$ a response $Y$ is modeled in dependence of an endogenous explanatory variable  $Z$. The proposed  estimator  is based on dimension reduction and additional thresholding. The minimax optimal rate of convergence of the estimator is derived  assuming that the structural function  and the representer of the linear functional belong to some ellipsoids which are in a certain sense linked to the conditional expectation operator of $Z$ given $W$. We illustrate these results by considering  classical smoothness assumptions.  
\end{abstract}

\begin{tabbing}
\>Linear Galerkin approach, Optimal rates of convergence,\\
\>Sobolev space,  finitely and infinitely smoothing operator.\\[.2ex]
\end{tabbing}

\section{Introduction}
Nonparametric instrumental regression models  have attract a growing attention recently in the econometrics and statistics literature 
(c.f. \cite{F03eswc}, \cite{DFR02}, \cite{NP03econometrica}, \cite{HallHorowitz2007} or \cite{BCK07econometrica} to name only a few). To be precise, these models deal with situations where the depends of a response $Y$ on the variation  of an endogenous vector $Z$ of explanatory variables  is characterized  by
  \begin{subequations}
	\begin{equation}
	\label{model:NP}
	Y = \varphi(Z) + U
	\end{equation}
for some  error term $U$, and  there exists an exogenous vector of instruments $W$ such that 
	\begin{equation}
	\label{model:NP2}
	\Ex [U | W] = 0 \; .
	\end{equation}
      \end{subequations}
The nonparametric relationship is thereby modeled by the structural function $\sol$. Typical examples leading to such situation are given by error-in-variable models, simultaneous equations or treatment models with endogeneous selection. However, it is worth noting that  in the presence of instrumental variables the model equations (\ref{model:NP}--\ref{model:NP2}) are the natural generalization of a standard parametric model (see, e.g., \cite{A74je}) to the nonparametric situation. This extension has been introduced first by \cite{F03eswc} and \cite{NP03econometrica}, while its identification has been studied  e.g. in \cite{CFR06handbook}, \cite{DFR02} and \cite{FJVB07}. It is interesting to note that recent applications and extensions of this approach include nonparametric tests of exogeneity (\cite{BH07res}), quantile regression models (\cite{HL07econometrica}), or semiparametric modeling (\cite{FJVB05}) to name but a few. 

The nonparametric estimation  of the structural function $\sol$ given a sample of $(Y,Z,W)$ has been intensively studied  in the literature. For example, 
\cite{AC03econometrica},  \cite{BCK07econometrica} or \cite{NP03econometrica}  consider sieve minimum distance estimator, while  \cite{DFR02}, \cite{HallHorowitz2005}, \cite{GS06} or \cite{FJVB07} study penalized least squares estimator. However,  as it has been noticed by \cite{NP03econometrica} and \cite{F03eswc}, the nonparametric estimation of the structural function $\sol$ generally leads to an ill-posed inverse problem. Precisely, consider the model equations (\ref{model:NP}--\ref{model:NP2}), then taking the conditional expectation  with respect to the instruments $W$ on both sides in equation (\ref{model:NP}) leads to the conditional moment equation:%
	\begin{equation}
	\label{model:EE}
	\Ex[Y | W] = \Ex[\sol(Z) | W].
	\end{equation}
Therefore, the estimation of the structural function $\sol$ is linked to the inversion of equation \eqref{model:EE}, which is under fairly mild assumptions not stable and hence  an ill-posed inverse problem (for a comprehensive review of inverse problems in econometrics we refer to \cite{CFR06handbook}). 

The instability of the conditional moment equation \eqref{model:EE}  essentially implies  that all  proposed estimators of the structural function $\sol$  have under reasonable assumptions  very poor rates of convergence. In other words, even relatively large sample sizes may not be of much help in accurately estimating the structural function $\sol$. In contrast, it might be possible to estimate certain local features of $\sol$, such as the value of a linear functional $\ell_\rep(\sol):=\Ex[\rep(Z)\sol(Z)]$  with respect to some given representer $h$,  at the usual parametric rate of convergence.  Take as an example the case of an endogenous  regressor $Z$  uniformly distributed  on $[0,1]$. In this situation  rather than estimating the structural function $\sol$ itself one may be interested in its average value $\int_a^b \sol(t)dt$ over a certain interval $[a,b]$ which equals the value $\ell_\rep(\sol)$ of a linear functional with representer given by the characteristic function $h=\1_{[a,b]}$. Then it is of interest to characterize the attainable accuracy of any estimator, for example, in terms of the mean squared error (MSE), which obviously depends on the representer $h$ and the conditions imposed on $\sol$. It is worth  noting, that the nonparametric estimation of the value of a linear functional  from Gaussian white noise observations is a subject of considerable literature (c.f. \cite{Speckman1979}, \cite{Li1982} or \cite{IbragimovHasminskii1984} in case of  direct observations, while in case of indirect observations we refer to \cite{DonohoLow1992}, \cite{Donoho1994} or \cite{GoldPere2000}  and references
therein). However, as far as we know this question has not yet been addressed  in  nonparametric instrumental regression, which in general is not a  Gaussian white noise model. 

The objective of this paper is the nonparametric estimation of the value $\ell_\rep(\sol)$ of a linear functional  based on an independent and identically distributed (i.i.d.) sample of $(Y,Z,W)$ obeying (\ref{model:NP}--\ref{model:NP2}). In this paper  we follow an often in the literature used  approach  to construct an estimator of the value of a linear functional. 
That is, we replace in  $\ell_\rep(\sol)$ the unknown structural function $\sol$ by an estimator. Therefore, let us first motivate the estimator of $\sol$ (for  its asymptotic properties we refer to \cite{Johannes2009}).  Suppose for a moment that the structural function can be developed by using only $m$ pre-specified functions $\basZ_1,\dotsc,\basZ_m$, say $\sol=\sum_{j=1}^m[\sol]_j \basZ_j$, where now  the coefficients $[\sol]_1,\dotsc,[\sol]_m$ are only unknown. 
Thereby, the conditional moment equation \eqref{model:EE} reduces to a multivariate linear conditional moment equation, that is,
$\Ex[Y | W] = \sum_{j=1}^m [\sol]_j\Ex[\basZ_j(Z)| W]$.  Notice that solving this  equation is a classical textbook problem in econometrics (c.f. \cite{PU99}). One popular approach is to replace the conditional moment equation by unconditional once. Therefore, given $m$ functions $\basW_1,\dotsc,\basW_m$ one may consider $m$ unconditional moment equations  in place of the multivariate conditional moment equation, that is, $\Ex[Y \basW_l(W)] = \sum_{j=1}^m [\sol]_j\Ex[\basZ_j(Z)\basW_l(W)]$, $l=1,\dotsc,m$. Notice that once the functions  $\{\basW_l\}_{l=1}^m$  are chosen all the unknown quantities in the unconditional moment equations can be straightforward estimated by replacing the theoretical expectation by its empirical counterpart. Moreover, a least squares solution of the estimated equation leads then under very mild assumptions to a consistent and asymptotic normal estimator of the parameter vector $([\sol]_j)_{j=1}^m$. Furthermore, the choice of the functions $\{\basW_l\}_{l=1}^m$ directly influences the asymptotic variance of the estimator and thus the question of optimal instruments arises (c.f. \cite{NP03econometrica}).  However, our objective is the estimation of the value of a linear functional. For simplicity suppose the regressor $Z$ is uniformly distributed  on $[0,1]$ and the linear functional is  given by the representer $h=\1_{[a,b]}$, that is, $\ell_h(\sol)=\int_a^b\sol(t)dt$. In case $\sol=\sum_{j=1}^m[\sol]_j \basZ_j$ the value of the linear functional writes $\ell_h(\sol)=\sum_{j=1}^m [h]_j [\sol]_j $ where the coefficients $[h]_j:=\int_a^b \basZ_j(t)dt,$ $1\leq j\leq m$, are known. A natural estimator of  $\ell_h(\sol)$ is then  defined by replacing the unknown coefficients $[\sol]_j$ by their least squares estimators. This approach is very simple and the estimator can be calculated with  most statistical software. However, it has a major default, since 
in most situations there is an infinite number of functions $\{\basZ_j\}_{j\geq1}$ and associated coefficients $([\sol]_j)_{j\geq1}$ needed to  develop the structural function $\sol$. The choice of the functions $\{\basZ_j\}_{j\geq1}$ reflects now the a priori information (such as smoothness) about the structural function $\sol$. However, if we consider also an infinite number of functions $\{\basW_l\}_{l\geq1}$  then for each $m\geq1$ we could still consider the least squares estimator described above. Notice, that the dimension $m$ plays here the role of a smoothing parameter and we may hope that  the estimator of the structural function $\sol$ (hence of the value $\ell_h(\sol)$) is also consistent as $m$ tends  suitably to infinity. Unfortunately, if $\sol_m:=\sum_{j=1}^m [\sol_m]_j \basZ_j$ denotes  a least squares solution  of the reduced unconditional moment equations, that is, the vector of coefficients $([\sol_m]_j)_{j=1}^m$ minimizes the quantity $\sum_{l=1}^m\{\Ex[Y \basW_l(W)] -\sum_{j=1}^m \beta_j\Ex[\basZ_j(Z)\basW_l(W)]\}^2$ over all $(\beta_j)_{j=1}^m$. Then,  $\sol_m$ converges to the true structural function as $m$ tends to infinity only under an additional assumption (defined below) on the basis $\{\basW_j\}_{j\geq1}$. In this paper we show under this additional assumption that in terms of the MSE a plug-in estimator of  $\ell_\rep(\sol)$  using  a least squares estimator of $\sol$  based on a dimension reduction together with an additional thresholding is consistent and  can attain optimal rates of convergences. It is worth to note that all the results in this paper are obtained without an additional smoothness assumption on the joint density of $(Y,Z,W)$. In fact we do even not impose that a joint density exists.

The paper is organized in the following way. In Section \ref{sec:lower} we introduce our basic assumptions and derive a lower bound for estimating the value of a linear functional  based on an i.i.d. sample obeying the model equations (\ref{model:NP}--\ref{model:NP2}). In Section \ref{sec:gen} under  certain moment assumptions  we show in terms of the MSE first  consistency  of the proposed estimator and second its minimax-optimality. We illustrate the general results in Section \ref{sec:sob} by considering classical smoothness assumptions. 
All proofs can be found in the Appendix.
\section{Complexity of local estimation: a lower bound.}\label{sec:lower}
\subsection{Basic model assumptions.}
It is convenient to rewrite the moment equation (\ref{model:EE}) in terms of an operator between Hilbert spaces. Let us first introduce the Hilbert Spaces
\begin{gather*}
L^2_Z=\{\phi:\R^p\to \R;\, \normV{\phi}^2_Z:=\Ex[\phi^2(Z)]<\infty\},\\
L^2_W=\{\psi:\R^q\to \R;\, \normV{\psi}^2_W:=\Ex[\psi^2(W)]<\infty\}
\end{gather*}
which are endowed with corresponding inner products $\skalarV{\phi,\tilde\phi}_Z=\Ex[\phi(Z)\tilde\phi(Z)]$, $\phi,\tilde\phi\in L^2_Z$, and $\skalarV{\psi,\tilde\psi}_W=\Ex[\psi(W)\tilde\psi(W)]$, $\psi,\tilde\psi\in L^2_W$, respectively. Then the conditional expectation of $Z$ given $W$ defines a linear operator $T\phi:=\Ex[\phi(Z)|W]$, $\phi\in L^2_Z$, which maps $L^2_Z$ into $L^2_W$.  Thereby  the moment equation (\ref{model:EE})  can be  written as
	\begin{equation}\label{bm:mequ}
	g:=\Ex[Y|W] = \Ex[\sol(Z)|W]=:T \varphi
	\end{equation}
where the function $g$ belongs to $L^2_W$.   Estimation of the structural function $\sol$ is thus linked with the
inversion of the conditional expectation operator $\op$ and, hence called an inverse problem. Moreover, we suppose throughout the paper that the operator $T$ is compact which is under fairly mild assumptions satisfied (c.f. \cite{CFR06handbook}). Consequently, unlike in a multivariate  linear instrumental regression model,  a continuous generalized inverse of $\op$ does not exist as long as the range of the operator $\op$ is an infinite dimensional subspace of $L^2_W$.  This corresponds to the setup of ill-posed inverse problems (with the additional difficulty that $\op$ is unknown and, hence has to be estimated). In what follows we always assume that there exists a unique  solution $\sol\in L^2_Z$ of equation (\ref{bm:mequ}),  i.e.,  $g$ belongs to the range $\cR(\op)$ of $\op$, and that the null space $\cN(\op)$ of $\op$ is trivial or equivalently $\op$ is injective (for a detailed discussion in the context of inverse problems see Chapter 2.1 in \cite{EHN00}, while in the special case of a nonparametric instrumental regression  we refer to \cite{CFR06handbook}). Furthermore, we suppose that the representer $h$ of the linear functional $\ell_\rep(\cdot):=\skalarV{\cdot,h}_Z$ of interest is an element of $L^2_Z$ as well. Then it is straightforward to see, that the value of the linear functional $\ell_\rep(\sol)$ is identified if and only if $\rep$ belongs to the orthogonal complement $\cN(\op)^\perp$ of the null space  $\cN(\op)$. Hence, for all $h\in L^2_Z$ the identification  is in particular guaranteed under the assumption of an injective  conditional expectation operator $\op$.


\subsection{Notations and regularity assumptions.}
In this section we show that the obtainable accuracy of any estimator of the value $\ell_\rep(\sol)$ of a linear functional can be essentially determined by additional regularity conditions imposed  on the structural function $\sol$, the representer $\rep$ and the conditional expectation operator $\op$. In this paper these conditions are characterized through different weighted norms in $L^2_Z$ with respect to a pre-specified orthonormal basis $\{\basZ_j\}_{j\geq1}$  in $L^2_Z$, which we formalize  now. 
Given a  strictly positive sequence of  weights $w:=(w_j)_{j\geqslant1}$ and  a constant $c>0$ we denote for all $r\in\R$ by $\cF_{w^r}^c$ the ellipsoid defined by \begin{equation}\label{bm:reg}
 \cF_{w^r}^c := \Bigl\{\phi\in L^2_Z: \sum_{j=1}^\infty w_j^r |\skalarV{\phi,\basZ_j}_Z|^2=:\normV{\phi}_{w^r}^2\leq c\Bigr\}.
\end{equation}
Furthermore, let $\cF_{w^r}:=\{\phi\in L^2_Z: \normV{\phi}_{w^r}^2<\infty\}$. It is worth noting, that  in case $w\equiv 1$ we have $\normV{\phi}_{w^r} = \normV{\phi}_{Z}$ for all $\phi\in L^2_Z$ and hence the set $\cF_w^c$ denotes an ellipsoid in $L^2_Z$ which  does not impose additional restrictions.

\paragraph{Minimal regularity conditions.}Let  $\bw:=(\bw_j)_{j\geqslant 1}$ and $\hw:=(\hw_j)_{j\geqslant1}$ denote  two sequences of weights. Then we suppose, here and subsequently,  that  the structural  function $\sol$ belongs to the ellipsoid $\cF_\bw^\br$ for some $\br>0$ and  that the representer $\rep$ of the linear functional $\ell_h$ is an element of the ellipsoid $\cF_\hw^\hr$ for some $\hr>0$. The ellipsoids $\cF_\bw^\br$  and $\cF_\hw^{\hr}$ capture all the prior information
(such as smoothness) about the unknown structural function $\sol$ and the given representer $\rep$ respectively.  Furthermore, as usual in the context of ill-posed inverse problems,  we specify  the mapping properties of the conditional expectation operator $\op$. 
Therefore, consider the sequence $(\normV{\op \basZ_j}_W)_{j\geqslant1}$, which  converges to zero since $\op$ is compact.  In what follows we impose restrictions on  the decay of this sequence. Denote by $\cT$ the set of all injective compact  operator mapping $L^2_Z$ into $L^2_W$.  Given  a sequence of weights $\tw:=(\tw_j)_{j\geqslant 1}$ and   $\td\geqslant 1$ we define the subset $\cT_{\tw}^\td$ of $\cT$  by 
\begin{equation}\label{bm:link}
\cT_{\tw}^{\td}:=\Bigl\{ T\in\cT:\quad   \normV{\phi}_{\tw}^2/d\leqslant \normV{T \phi}_W^2\leqslant {\td}\, \normV{\phi}_{\tw}^2,\quad \forall \phi \in L^2_Z\Bigr\}.
\end{equation}
Notice  that  for all  $T\in\cT_{\tw}^\td$  it follows that\footnote{We write $a\asymp_\td b$ if  $\td^{-1}\leqslant b/a\leqslant \td$.}  $\normV{T\basZ_j}_W^2\asymp_{{\td}} \tw_j$. Hence, the sequence $(\tw_j)_{j\geqslant 1}$ has to be strictly positive since $T$ is injective. Furthermore, let us denote by $T^*:L^2_W\to L^2_Z$ the adjoint of $T$ which  satisfies  $T^*\psi=\Ex[\psi(W)|Z]$. If now $ T\in\cT$ and $\{\lambda_j,\basZ_j\}_{j\geq 1}$ is an eigenvalue decomposition of $T^*T$. Then the condition $T\in\cT_{\tw}^\td$ is satisfied if and only if $\lambda_j\asymp_d \tw_j$. In other words, in this situation the sequence $\tw$ specifies the decay of the eigenvalues of $T^*T$.  In what follows all the results are derived  under  regularity conditions on the structural function $\sol$, the representer $\rep$ and   the conditional expectation operator $\op$ described through  the sequence $\bw$, $\hw$ and $\tw$ respectively.  However, we provide  below  illustrations of these conditions   by assuming a \lq\lq regular decay\rq\rq\ of these sequences.   The next assumption summarizes our minimal regularity conditions on these sequences.
\begin{assumption}\label{ass:reg} Let  $\bw:=(\bw_j)_{j\geqslant 1}$, $\hw:=(\hw_j)_{j\geqslant 1}$ and  $\tw:=(\tw_j)_{j\geqslant 1}$ be strictly positive sequences of weights  with   $\bw_1= 1$, $\hw_1= 1$ and  $\tw_1= 1$ such that   $\bw$ and $\hw$ are non decreasing and  $\tw$ is non increasing. Furthermore,  there exists a constant $\hwtwD\geqslant 1$ such that $\tw_m\sup_{1\leqslant j\leqslant m}\{\tw_j^{-1}\hw_j^{-1}\}\leqslant \hwtwD\max(\hw_m^{-1},\tw_m)$ for all $m\in\N$. \end{assumption}
We shall stress  that  $\cF_\bw^\br$  is just an ellipsoid in $L^2_Z$ in case $\bw\equiv 1$, hence in this situation there is not  an additional regularity condition on the structural function $\sol$ imposed.  Furthermore,  the last condition in Assumption \ref{ass:reg} is obviously satisfied  with $\hwtwD=1$ if the sequence $(\tw_j\hw_j)_{j\geq1}$ is either monotonically decreasing or increasing.

\subsection{The lower bound.}
In the proof of the next theorem we show that an  one-dimensional subproblem captures the full difficulty in estimating a linear functional in  nonparametric instrumental regression.  In other words, there exist  two sequences of structural functions $\sol_{1,n},\sol_{2,n}\in \cF_\bw^\br$, which are statistically not consistently distinguishable, and a sequence of representer $h_n \in \cF_\hw^\hr$ such that  $|\ell_{h_n}(\varphi_{1,n})-\ell_{h_n}(\varphi_{2,n})|^2\geq C \delta_n$, where $\delta_n$ is the optimal rate of convergence. Moreover, we obtain the following lower bound  under the additional assumption that  there exist error terms $U_{i,n}$, $i=1,2$, such that the conditional distribution of $\sol_{i,n}-T\sol_{i,n}+ U_{i,n}$ given the instrument $W$  is Gaussian with mean zero and variance one.  A similar assumption has recently been used by \cite{ChenReiss2008} in order to derive a lower bound for the estimation of the structural function $\sol$ itself. In particular the authors show that in opposite to the present work an one-dimensional subproblem is not sufficient to describe the full difficulty in estimating $\sol$.
\begin{theo}\label{res:lower}\dr Assume an  $n$-sample of $(Y,Z,W)$ from the model  (\ref{model:NP}--\ref{model:NP2}) with error term $U$ belonging to 
 $\cU_\sigma:=\{ U : \Ex U|W=0 \mbox{ and }\Ex U^4|W \leq \sigma^4\}$, $\sigma>0$. Let $\bw$, $\hw$ and $\tw$ be sequences 
satisfying Assumption \ref{ass:reg}. Suppose that  the conditional expectation operator $T$ associated $(Z,W)$  belongs to $\cT_{\tw}^\td$, $\td\geq 1$, and that $\sup_{j\geq} \Ex [e_j^4(Z)|W]\leq \eta$,  $\eta\geq1$.  Let $\kstar:=\kstar(n)\in\N$ and $\dstar:=\dstar(n)\in\R^+$   be chosen such that for some $\triangle\geqslant 1$ 
\begin{equation}\label{res:lower:def:md}
1/\triangle\leqslant  \frac{\bw_{\kstar}}{n\, \tw_{\kstar}} \leqslant \triangle\quad\text{ and }\quad \dstar:=\bw_{\kstar}^{-1}\hw_{\kstar}^{-1}.
\end{equation}
If in addition $\sigma$ is sufficiently large  then  for any estimator $\breve{\ell}$   we have  
\begin{equation*}   \sup_{U\in \cU_\sigma} \sup_{\sol \in\cF_\bw^{\br}} \sup_{\rep \in \cF_\hw^{\hr}} \left\{ \Ex|\breve{\ell}-\ell_\rep(\sol)|^2\right\}\geqslant  \max \Bigl(\dstar\,,\,\frac{1}{n}\Bigr) \,\frac{\hr}{4\,{\triangle}}\,\min \Bigl( \frac{\sigma_0^2}{2\, \td}\,,\,  \frac{\br}{\triangle}\Bigr).
\end{equation*}
\end{theo}
\begin{rem}In the last theorem the additional moment condition $\sup_{j\geq1} \Ex [e_j^4(Z) |W]\leq \eta$  is obviously satisfied if the basis functions $\{\basZ_j\}$ are uniformly bounded (e.g. the trigonometric basis considered in Section \ref{sec:sob}). However, if $V$ denotes Gaussian random variable with mean zero and variance one, which is independent of $(Z,W)$, then the additional moment condition ensures that  for all structural functions of the form $\sol= a\cdot\basZ_j \in\cF_\bw^{\br}$ with $j\geq 1$  and  $a\in\R$,  the  error term $U:= V -\sol(Z) + [T\sol](W)$ belongs to $\cU_\sigma$  for all sufficiently large $\sigma$. This specific case is only needed to simplify the calculation of the distance  between distributions corresponding to different structural functions. On the other hand, below we derive an upper bound assuming  that 
the error term $U$ belongs to $\cU_\sigma$ and that the joint distribution of $(Z,W)$ fulfills additional moment conditions. Obviously in this situation Theorem \ref{res:lower} provides a lower bound for any estimator as long as $\sigma$ is sufficiently large.  Furthermore, it is worth noting that  this lower bound  tends only to zero if  $(\hw_{j}\bw_{j})_{j\geq 1}$ is a divergent sequence. In other words, in case $\bw\equiv 1$, i.e., without  any additional restriction on $\sol$,  consistency  of an estimator of  $\ell_\rep(\sol)$ uniformly over all  $\sol\in \cF_\bw^\br$ is only possible under restrictions on the representer $\rep \in\cF_\hw^\hr$, that is, $\hw$ is a divergent sequence. This obviously reflects the ill-posedness of the underlying inverse problem. Finally, it is important  to note that the regularity conditions imposed  on the structural function $\sol$, the representer $\rep$ and the conditional expectation operator $T$ involve only the basis $\{\basZ_j\}_{j\geq1}$ in $L^2_Z$. Therefore, the lower bound derived in Theorem \ref{res:lower} does  not capture the influence of the basis $\{\basW_l\}_{l\geq1}$ in $L^2_W$ used to construct the estimator. In other words, an estimator of the value $\ell_\rep(\sol)$ can only attain this lower bound if  $\{\basW_l\}_{l\geq1}$  is appropriate chosen. 
\hfill$\square$\end{rem}

\section{Minimax-optimal local estimation: the general case.}\label{sec:gen}
\subsection{Estimation by dimension reduction and thresholding.}
In addition to the basis $\{\basZ_j\}_{j\geq1}$ in $L^2_Z$ considered in the last section we introduce now also a second basis $\{\basW_l\}_{l\geq1}$ in $L^2_W$. We derive in this section the asymptotic properties of the estimator under minimal assumptions on those basis. Precisely, we show first consistency of the proposed estimator under fairly mild additional moment assumptions. In particular, we do not impose any regularity assumption on both the structural function $\sol$ and the representer $\rep$.  In a second step we suppose that the structural function $\sol$ and the representer $\rep$ belong to some ellipsoid $\cF_\bw^\br$ and $\cF_\hw^\hr$ respectively, and that the conditional expectation satisfies a link condition, i.e., $T\in\cTdw $. Furthermore, we introduce an additional condition linked to the basis $\{\basW_l\}_{l\geq1}$. Then under stronger moment conditions  we show that the proposed estimator attains the lower bound derived in the last section. However, all these results are illustrated in the next section by considering classical smoothness assumptions. 

\paragraph{Matrix and operator notations.} Given $m\geq 1$, $\cE_m$ and $\cF_m$ denote the subspace of $L^2_Z$ and $L^2_W$ spanned by the functions $\{e_j\}_{j=1}^m$ and $\{f_l\}_{l=1}^m$, respectively. $E_m$ and $E_m^\perp$ (resp. $F_m$ and $F_m^\perp$) denote the orthogonal projections on  $\cE_m$ (resp. $\cF_m$) and its orthogonal complement $\cE_m^\perp$ (resp. $\cF_m^\perp$), respectively. Given an operator (matrix) $K$, $\normV{K}$ denotes its operator norm . The inverse operator (matrix) of $K$ is denoted by $K^{-1}$, the adjoint (transposed) operator (matrix) of $K$ by $K^*$. The identity operator (matrix) is denoted by $I$. $[\phi]$, $[\psi]$ and $[K]$ denote the (infinite) vector and matrix of the function $\phi\in L^2_Z$, $\psi\in L^2_W$ and the operator $K:L^2_Z\to L^2_W$ with the entries $[\phi]_{j}=\skalarV{\phi,e_j}$, $[\psi]_{l}=\skalarV{\psi,f_l}$ and $[K]_{l,j}=\skalarV{Ke_j,f_l}$, respectively. The upper $m$ subvector and $m\times m$ submatrix
of $[\phi]$, $[\psi]$ and $[K]$  is denoted by $[\phi]_{\um}$, $[\psi]_{\um}$ and $[K]_{\um}$, respectively. Note, that $[K^*]_{\um}=[K]_{\um}^*$. The diagonal   matrix with entries $v$ is denoted by $\Diag(v)$.   Clearly, $[E_m \phi]_{\um} =[\phi]_{\um}$ and if we restrict $F_m K E_m$ to an  operator from $\cE_m$ into $\cF_m$, then it has the matrix $[K]_{\um}$. 

Consider the conditional expectation operator $T$ associated to the regressor $Z$ and the instrument $W$. If  $[\basZ(Z)]$ and $[\basW(W)]$ denote the infinite random vector  with entries $\basZ_j(Z)$ and $\basW_j(W)$ respectively, then $[T]_{\um}=\Ex [\basW(W)]_{\um}[\basZ(Z)]_{\um}^*$ which is throughout the paper  assumed   to be non singular for all $m\geq 1$ (or, at least for large enough $m$), so that $[T]_{\um}^{-1}$ always exists. Note that it is a nontrivial problem
to determine when such an assumption holds (see e.g. \cite{EfromovichKoltchinskii2001} and references therein). Under this assumption the notation $T_m^{-1}$ is used for the operator from $L^2_W$ into $L^2_Z$, whose matrix  in the basis $\{e_j\}_{j\geq1}$ and $\{f_l\}_{l\geq1}$ has the entries $([T]_{\um}^{-1})_{j,l}$ for $1\leqslant j,l\leqslant m$ and zeros otherwise.

\paragraph{Definition of the estimator.}Let $(Y_1,Z_1,W_1),\dotsc,(Y_n,Z_n,W_n)$ be an i.i.d. sample of $(Y,Z,W)$. 
Since $[T]_{\um}=\Ex [\basW(W)]_{\um}[\basZ(Z)]_{\um}^*$ and $[g]_{\um}=\Ex{Y[\basW(W)]_{\um}}$ we constuct estimators by using their empirical counterparts, that is,
\begin{equation}\label{form:def:est:T:g}
[\widehat{T}]_{\um}:= (1/n)\sum_{i=1}^n [\basW(W_i)]_{\um}[\basZ(Z_i)]_{\um}^*  \quad\mbox{ and }\quad[\widehat{g}]_{\um}:= (1/n)\sum_{i=1}^n Y_i[\basW(W_i)]_{\um}.
\end{equation}
Then the  estimator of the linear functional $\ell_\rep(\sol)$ is defined  by
\begin{equation}\label{gen:def:est}
\widehat{\ell}_\rep:=
\left\{\begin{array}{lcl} 
[\rep]_{\um}^t[\hop]_{\um}^{-1} [\widehat{g}]_{\um}, && \mbox{if $[\hop]_{\um}$ is nonsingular and }\normV{[\hop]^{-1}_{\um}}\leq \alpha,\\
0,&&\mbox{otherwise},
\end{array}\right.
\end{equation}
where  the dimension parameter $m=m(n)$  and the threshold $\alpha=\alpha(n)$  have to tend  to infinity as the sample size $n$  increases. In fact, the estimator $\widehat{\ell}_h$ is obtained from the linear functional $\ell_\rep(\sol)$ by replacing the unknown structural function $\sol$ by an estimator proposed by \cite{Johannes2009}, which takes its inspiration in  the  linear Galerkin approach coming from the inverse problem community  (c.f. \cite{EfromovichKoltchinskii2001} or \cite{HoffmannReiss04}). 
\subsection{Consistency.}
We start by providing minimal conditions used  to proof consistency of the estimator. More specific, we formalize first additional moment assumptions  on the basis under consideration.
\begin{assA}\label{ass:A1}
The joint distribution of $(Z,W)$ satisfies  $\sup_{j\in N}\Ex [f_j^2(W)|Z]\leqslant \eta^2$ and $\sup_{j,l\in\N} \Var (e_j(Z)f_l(W))\leqslant \eta^2$ for some $\eta\geq 1$.
\end{assA}
It is worth noting that the Assumption \ref{ass:A1} is always fulfilled in case  both basis are uniformly bounded (e.g. in case of the trigonometric basis considered in Section \ref{sec:sob}). The next assertion summarizes  our minimal conditions to ensure consistency of the estimator ${\widehat\ell_\rep}$ introduced  in (\ref{gen:def:est}). 
  \begin{prop}\label{res:gen:prop:cons}\dr Assume an $n$-sample of $(Y,Z,W)$ from the model (\ref{model:NP}--\ref{model:NP2}). Suppose that the  error term $U$ satisfies $\Ex U^2|W\leqslant \sigma^2$ with  $\sigma>0$ and  that the joint distribution of $(Z,W)$ fulfills Assumption \ref{ass:A1}. Let  ${\widehat\ell_\rep}$ be defined with  dimension  $m:=m(n)$ and threshold $\alpha:=\alpha(n)$ satisfying $\alpha\geqslant 2\normV{[T]^{-1}_{\um}}$ and as  $n\to\infty$ that $1/m=o(1)$ and  $m^2 \alpha^2=o(n)$. If in addition  $\sup_{m\in\N}\normV{T^{-1}_{m} F_m T E_m^\perp}<\infty$, then we have
 $\Ex|\widehat\ell_\rep-{\ell_\rep(\sol)}|^2=o(1)$ as  $n\to\infty$. 
  \end{prop}
The last result shows  consistency of the estimator without an a priori regularity assumption on both the structural function $\sol$ and the representer $\rep$. However, consistency is only obtained under the condition $\sup_{m\in\N}\normV{T^{-1}_{m} F_m T E_m^\perp}<\infty$, which is known to be necessary to ensure $L^2$-convergence  of the least squares  solution $\sol_m=\sum_{j=1}^m[\sol_m]_{j}\basZ_j$ with $[\sol_m]_{\um}=[\op]_{\um}^{-1}[g]_{\um}$ to the structural function  $\sol$ as $m\to \infty$. Notice that this condition involves now also the basis $\{\basW_l\}_{l\geq1}$ in $L^2_W$. In what follows we introduce an alternative but stronger condition to guarantee the $L^2$-consistency which  extends the link condition \eqref{bm:link}, that is, $T\in\cTdw$.
We denote by   
$\cTdDw$ for some  $\tD\geq \td$ the subset of $\cTdw $ given by
\begin{equation}\label{bm:link:gen}
\cTdDw:=\Bigl\{ T\in \cTdw :\quad \sup_{m\in\N}\normV{[\Diag(\tw)]^{1/2}_{\um}[T]^{-1}_{\um}}^2\leqslant \tD\Bigr\}.
\end{equation}
\begin{rem}If $\{\sqrt{\lambda}_j,e_j, f_j\}_{j\geq1}$ is a singular value decomposition of $\op\in\cT$ then for all $m\geq 1$ the matrix $[\op]_{\um}$ is diagonalized with diagonal entries $[\op]_{j,j}=\sqrt{\lambda}_j$, $1\leq j\leq m$. Therefore, the  link condition \eqref{bm:link} holds true, that is, $T\in\cTdw $, if and only if $ \lambda_j\asymp_{{d}}\tw_j$ for all $j\in\N$. Moreover,  it is easily seen that $\sup_{m\in\N}\normV{[\Diag(\tw^{1/2})]_{\um}[T]^{-1}_{\um}}^2\leqslant \td$ and hence the  extended link condition \eqref{bm:link:gen} is fulfilled, that is, $T\in \cTdDw$  for all $D\geq d$. Furthermore,  the extended link condition equals the link condition ($\cTdw =\cTdDw$ for suitable $\tD>0$), if  $[T]$ is only   a small perturbation of $\Diag(\tw^{1/2})$ or if $T$ is  strictly positive   (for a detailed discussion we refer to \cite{EfromovichKoltchinskii2001}) and \cite{CardotJohannes2008} respectively). 

We shall stress  that once both basis $\{\basZ_j\}_{j\geq1}$ and $\{\basW_l\}_{l\geq1}$ are specified the extended link condition \eqref{bm:link:gen} restricts the class of joint distributions of $(Z,W)$ to those for which the least squares solution $\sol_m$ is $L^2$-consistent. Moreover, it is shown in \cite{Johannes2009},  that under the extended link condition  a least squares estimator of $\sol$  based on a dimension reduction together with an additional thresholding can attain  minimax-optimal rates of convergence. In this sense, given a  joint distribution of $(Z,W)$  a basis $\{\basW_l\}_{l\geq1}$ satisfying the extended link condition  can be interpreted as optimal instruments.  However, for each pre-specified basis $\{\basZ_j\}_{j\geq1}$  we can theoretically construct a basis $\{\basW_l\}_{l\geq1}$ such that the extended link condition is not a stronger restriction than the link condition \eqref{bm:link}. To be more precise, if $T\in\cTdw  $, which involves only the basis $\{\basZ_j\}_{j\geq1}$, then it is not hard to see that the fundamental inequality of \cite{Heinz51} implies   $ \normV{ (T^*T)^{-1/2} \basZ_j}^2 \asymp_d \tw_j^{-1}$ for all $j\geq1$. Thereby, the function $(T^*T)^{-1/2} \basZ_j$ is an element of $L^2_Z$ and hence there exists $\basW_j:= T (T^*T)^{-1/2} \basZ_j \in L^2_W$, $j\geq1$. Then it is easily checked that  $\{\basW_l\}_{l\geq1}$ is an orthonormal system and moreover a basis of the closure of the range $\cR(T)$ of $T$. Hence by taking any basis of the orthogonal complement $\cR(T)^\perp$ of $\cR(T)$ we may complete the orthonormal set $\{\basW_l\}_{l\geq1}$ to become a basis of $L^2_W$. 
%
Then it is straightforward to see that $[T]_{\um}$ is symmetric and moreover strictly positive, since $\skalarV{ \op\basZ_j,\basW_l}_W= \skalarV{ \op\basZ_j,\op (\op^*\op)^{-1/2} \basZ_l}_W=\skalarV{ (\op^*\op)^{1/2}\basZ_j, \basZ_l}_Z $ for all $j,l\geq 1$. Thereby, we can  apply Lemma A.3 in \cite{CardotJohannes2008} which gives $\cTdw= \cTdDw$ for all sufficiently large $D$.
%
\hfill$\square$
\end{rem}

Under the extended link condition \eqref{bm:link:gen}, that is, $T\in \cTdDw$,  the next assertion summarizes  minimal conditions to ensure consistency. 
 \begin{coro}\label{res:gen:coro:cons}\dr Let  the assumptions of Proposition \ref{res:gen:prop:cons} be satisfied and assume in addition that $T\in\cTdDw$.
If ${\widehat\ell_\rep}$ is defined with threshold $\alpha= 2\sqrt{D/\tw_m}$ and dimension $m:=m(n)$  such that $ m^2/(n\tw_m)=o(1)$ and $1/m=o(1)$. Then we have
 $\Ex|\widehat\ell_\rep-{\ell_\rep(\sol)}|^2=o(1)$, as  $n\to\infty$. 
\end{coro}
\subsection{The upper bound.}
The last assertions show that the estimator $\widehat\ell_\rep$   defined in \eqref{gen:def:est}  is consistent without any additional  regularity conditions both on   structural function  and  representer. The following theorem provides  now an upper bound if these conditions are given through  ellipsoids $\cF_\bw^\br$ and $\cF_\hw^\hr$ for the structural function and the representer respectively together with an extended link condition \eqref{bm:link:gen} for the conditional expectation operator $T$.   Furthermore,  the result  is derived under stronger moment conditions  on the basis, more specific, on  the  random vector  $[\basZ(Z)]$ and $[\basW(W)]$, which we formalize first. 
\begin{assA}\label{ass:A2}
There exists $\eta\geqslant 1$ such that the joint distribution of $(Z,W)$ satisfies  
\begin{itemize}
\item[(i)] $\sup_{j\in N}\Ex [e_j^2(Z)|W]\leqslant \eta^2$ and $\sup_{l\in N}\Ex [f_l^4(W)]\leqslant \eta^4$;
\item[(ii)] $\sup_{j,l\in\N} \Var (e_j(Z)f_l(W))\leqslant \eta^2$ and\\ $\sup_{j,l\in\N} \Ex| e_j(Z)f_l(W)- \Ex [e_j(Z)f_l(W)]|^8\leqslant 8!\eta^{6} \Var (e_j(Z)f_l(W))$.
\end{itemize} 
\end{assA}
It is worth noting that again any joint distribution  of $(Z,W)$ satisfies Assumption \ref{ass:A2} for sufficiently large $\eta$   if  the basis $\{e_j\}_{j\geq1}$ and $\{f_l\}_{l\geq1}$ are uniformly bounded. Here and subsequently, we write $a_n\lesssim b_n$  when there exists $C>0$ such that  $a_n\leqslant C\, b_n$  for all sufficiently large $ n\in\N$ and  $a_n\sim b_n$ when $a_n\lesssim b_n$ and $b_n\lesssim a_n$ simultaneously.

\begin{theo}\label{res:upper}\dr Assume an $n$-sample of $(Y,Z,W)$ from the model (\ref{model:NP}--\ref{model:NP2}) with  error term $U \in \cU_\sigma$, $\sigma>0$. Suppose that the  joint distribution of $(Z,W)$ fulfills Assumption \ref{ass:A2} for some $\eta\geq1$
and that the associated conditional expectation operator $T\in\cTdDw$, $\td,\tD\geq1$, where the sequences
 $\bw$, $\hw$ and $\tw$  satisfy Assumption \ref{ass:reg}.  Let $\kstar:=\kstar(n)$ and  $\dstar:=\dstar(n)$ be such that \eqref{res:lower:def:md} holds for some $\triangle\geq1$. Consider the estimator $\widehat\ell_\rep$   with dimension $m:=\kstar$ and threshold $\alpha^2:= n\max(1, 4\, {\tD\,\triangle/  \bw_{m}})$. If in addition $\Gamma:=\sum_{j\in\N}\bw_j^{-1}<\infty$, then we have
\begin{multline*}
\sup_{\sol \in \cF_\bw^{\br}}\sup_{\rep \in \cF_\hw^{\hr}} \Ex|\widehat\ell_\rep- \ell_\rep(\sol)|^2\lesssim  \td\,\tD^2\,\hwtwD\,\triangle\, \eta^4\,\{\sigma^{2}+\td\tD\Gamma\}\,\rho\,\tau\,\\
\hfill\Bigl\{1+  \kstar^3/\bw_{\kstar} + \kstar^3 \Bigl|P\Bigl(\normV{[\widehat T]_{\ukstar}-[T]_{\ukstar}}^2> \tw_{\kstar}/(4\tD)\Bigr)\Bigr|^{1/4} \Bigr\}\hfill\\\Bigl\{ \max\Bigl(\dstar,1/n\Bigr) + P\Bigl(\normV{[\widehat T]_{\ukstar}-[T]_{\ukstar}}^2> \tw_{\kstar}/(4\tD)\Bigr)\Bigr\}.
\end{multline*}
\end{theo}
We shall stress that the bound in the last theorem is  non asymptotic. However, it does not establish the optimality of the estimator compared with the lower bound in Theorem \ref{res:lower}. But, the bound in Theorem \ref{res:upper} can be improved by imposing a moment condition stronger than Assumption \ref{ass:A2}. To be more precise, consider the centered random variable $e_j(Z)f_l(W)- \Ex [e_j(Z)f_l(W)]$. Then  Assumption \ref{ass:A2} (ii) states that  its $8$th moment  is uniformly bound over $j,l\in\N$. In the next Assumption we suppose that these random variables satisfy uniformly Cramer's condition, which is known to be sufficient to obtain an exponential bound for their large deviations (c.f. \cite{Bosq1998}).
\begin{assA}\label{ass:A3}
There exists $\eta\geqslant 1$ such that the joint distribution of $(Z,W)$ satisfies  Assumption \ref{ass:A2} and in addition 
\begin{itemize}
\item[(iii)] $\sup_{j,l\in\N} \Ex| e_j(Z)f_l(W)- \Ex [e_j(Z)f_l(W)]|^k\leqslant \eta^{k-2} k! \Var(e_j(Z)f_l(W))$, $k=3,4,\dotsc$.
\end{itemize} 
\end{assA}
It is well-known that Cramer's condition is in particular fulfilled if the random variable  $e_j(Z)f_l(W)- \Ex [e_j(Z)f_l(W)]$ is bounded. Hence 
in case the basis $\{e_j\}_{j\geq1}$ and $\{f_l\}_{l\geq1}$ are uniformly bounded it follows again that any joint distribution  of $(Z,W)$ satisfies Assumption \ref{ass:A3} for sufficiently large $\eta$.  On the other hand, in Lemma \ref{app:gen:upper:l4} in the Appendix  we show that  Assumption \ref{ass:A3}  implies an exponential bound on  the large deviation probability $P(\normV{[\widehat T]_{\um}-[T]_{\um}}^2> \upsilon_m/(4\tD))$. Thereby, if  the  sequences $\bw$, $\hw$ and $\tw$   have the following additional properties
\begin{equation}\label{gen:upper:varphi:cond}
\kstar^2(\log \bw_{\kstar})\bw_{\kstar}^{-1}=o(1), \;\; \kstar^2(\log  \min(\tw^{-1}_{\kstar},\hw_{\kstar}))\bw^{-1}_{\kstar}=o(1),\;\;   \kstar^3\bw^{-1}_{\kstar}=O(1)\quad \mbox{  as }n\to\infty, 
\end{equation}
where  $\kstar:=\kstar(n)$ and $\dstar:=\dstar(\kstar)$ are given by \eqref{res:lower:def:md}, then the large deviation probability  tends to zero more quickly than $\max(\dstar,1/n)$. In this situation it is not hard to see that $\max(\dstar,1/n)$ is the order of the upper bound given in  Theorem \ref{res:upper}.  Hence, the rate $\max(\dstar,1/n)$ is optimal and  $\widehat\ell_\rep$   is  minimax-optimal, which is summarized in the next assertion.

\begin{theo}\label{res:upper:A3}\dr Suppose that the assumptions of Theorem \ref{res:upper} are satisfied.  In addition assume that the  joint distribution of $(Z,W)$ fulfills Assumption \ref{ass:A3} and that the sequences $\bw$, $\hw$ and $\tw$   have the properties \eqref{gen:upper:varphi:cond}. Then, we have
\begin{equation*}
\sup_{\sol \in \cF_\bw^{\br}}\sup_{\rep \in \cF_\hw^{\hr}} \Ex|\widehat\ell_\rep- \ell_\rep(\sol)|^2\lesssim  \td\,\tD^2\,\hwtwD\,\triangle\, \eta^4\,\{\sigma^{2}+\td\tD\Gamma\}\,\rho\,\tau\,\max(\dstar,n^{-1}).
\end{equation*}
\end{theo}

\begin{rem}\label{rem:upper:A3} It is worth noting  that the bound in the last result is again non asymptotic. Furthermore, from  Theorem \ref{res:lower} and  \ref{res:upper:A3} follows that the estimator $\widehat{\ell}_h$ attains the optimal rate $\max(\dstar,n^{-1})$ (hence is minimax-optimal) for all sequences  $\bw$, $\hw$ and $\tw$  satisfying both the minimal  regularity  conditions summarized in Assumption \ref{ass:reg} and the additional properties \eqref{gen:upper:varphi:cond}. We shall emphasize the interesting influence of  the sequences $\bw$, $\hw$ and $\tw$. As we see from Theorem \ref{res:lower} and \ref{res:upper:A3}, if the sequence  $\tw$ decreases more quickly to zero then the obtainable optimal rate of convergence decreases. On the other hand, a faster increasing sequence $\bw$ or  $\hw$ leads   to a faster optimal rate. In other words, as expected, values of a linear functional  given by a  structural function or representer satisfying a stronger regularity condition can be estimated faster. 

Note furthermore, if the eigenfunctions of the operator $\op$ are given by $\{\basZ_j\}_{j\geq1}$ and $\{\basW_l\}_{l\geq1}$, then $\op\in\cTdDw$ holds if and only if the corresponding singular values $[T]_{jj}=\skalarV{\op\basZ_j,\basW_j}$, $j\geq1$, satisfy $[T]_{jj}^2\asymp_d\tw_j$. Hence, in this situation the optimal rate obtained in the last  assertion is linked to the decay of the singular values of $\op$. However, the set  $\cTdDw$ contains also operators with eigenfunctions not given by $\{\basZ_j\}_{j\geq1}$ and $\{\basW_l\}_{l\geq1}$. Then their corresponding eigenvalues may decay far slower than the sequence of weights $\tw$. Moreover, it is straightforward to show, that  by using  a projection onto the basis $\{\basZ_j\}_{j\geq1}$ and $\{\basW_l\}_{l\geq1}$ instead of their eigenfunctions, the obtainable rate of convergence given in Theorem \ref{res:upper:A3}  may be far slower than the rate obtained by using the eigenfunctions (see e.g. \cite{JohannesSchenk2009} in the context of functional linear model). However, the rate in Theorem \ref{res:upper:A3} is optimal since the eigenfunctions are generally unknown. 

Finally, since the sequence $\bw$ increases it follows that   in Theorem \ref{res:upper} and hence also in  Theorem \ref{res:upper:A3} for all large enough $n$  the threshold $\alpha=n$  is used  to construct the estimator  ${\widehat\ell_\rep}$. On the other hand, the choice of the dimension $m$  depends on the sequences $\bw$ and $\tw$ characterizing the regularity conditions imposed on the structural function and the conditional expectation operator which are in practice not known. Building data driven rules that can permit to choose automatically the value  of $m$  is certainly a topic that deserves further attention and one promising direction is to adapt the selection technique proposed in  \cite{EfromovichKoltchinskii2001}, \cite{GoldPere2000} or \cite{Tsybakov2000}.\hfill$\square$\end{rem}

\section{Minimax-optimal estimation under classical smoothness assumptions.}\label{sec:sob}
In this section we shall describe the prior information about the unknown structural function $\sol$ and the given representer $\rep$ by their level of smoothness. In order to simplify the presentation we follow  \cite{HallHorowitz2005} (where also a more detailed discussion of this assumption can be found), and suppose that the marginal distribution of the scalar regressor $Z$ and the scalar instrument $W$ are uniformly distributed on the interval $[0,1]$. It is worth noting that all the results below can be straightforward extended  to the multivariate case. However, in the univariate case it follows that both  Hilbert spaces $L^2_Z$ and $L^2_W$ equal $L^2[0,1]$, which is endowed with the usual norm $\norm$ and inner product $\skalar$. 

In the last sections we have seen that the choice of the basis $\{e_j\}_{j\geq1}$ is directly linked to the a priori assumptions we are willing to impose on the structural function and the representer. In case of classical smoothness assumptions it is natural to consider the trigonometric basis%
\begin{equation}\label{bm:def:trigon}
e_{1}:\equiv1, \;e_{2j}(s):=\sqrt{2}\cos(2\pi j s),\; e_{2j+1}(s):=\sqrt{2}\sin(2\pi j s),s\in[0,1],\; j\in\N,\end{equation}
which can be realized as follows. Let us introduce the  Sobolev space of periodic functions $\cW_r$, $r\geqslant0$, which for integer $r$ is given by 
  \begin{equation*}
 \cW_{r}=\Bigl\{f\in H_{p}: f^{(j)}(0)=f^{(j)}(1),\quad j=0,1,\dotsc,r-1\Bigr\},
 \end{equation*}
 where  $H_{r}:= \{ f\in L^2[0,1]:  f^{(r-1)}\mbox{ absolutely continuous }, f^{(r)}\in L^2[0,1]\}$  is a Sobolev space.\linebreak If we consider now   $\cF_{w^r}$ given in \eqref{bm:reg} with weight sequence  $w_1=1$, $w_{j}=|j|^{2},$ $j\geqslant2$, and trigonometric basis $\{e_j\}$, then it is well-known that  the subset $\cF_{w^r}$ coincides with the Sobolev space of periodic functions $\cW_r$ (c.f. \cite{Neubauer1988,Neubauer88}, \cite{MairRuymgaart96} or \cite{Tsybakov04}). Therefore, let us denote by  $\cW_{r}^c:= \cF_{w^r}^c$, $c>0$ an ellipsoid in the Sobolev space $\cW_r$. We use  in case $r=0$ again the convention that $\cW_r^c$ denotes an ellipsoid in $L^2[0,1]$. In the rest of this section we suppose that the unknown structural function $\sol$ and the given representer $\rep$  are $p\geq 0$ and $s\geq 0$ times differentiable, respectively. More precisely, the prior information about  $\sol$ and  $\rep$ are characterized by the Sobolev ellipsoid $\cW_p^\br$, $\br>0$, and $\cW_s^\hr$, $\hr>0$, respectively. 

Furthermore, to illustrate the general results in Section \ref{sec:gen} we consider two special cases describing a \lq\lq regular decay\rq\rq\ of the sequence $\tw$, which characterizes the mapping properties of the associated conditional expectation operator. Precisely,  we assume in the following the sequence $\tw$ to be either  polynomially decreasing, i.e., $\tw_1=1$ and $\tw_j  = |j|^{-2a}$, $j\geqslant 2$,  or   exponentially decreasing, i.e.,  $\tw_1=1$ and $\tw_j  = \exp(-|j|^{2a})$, $j\geqslant 2$, for some $a>0$. In the polynomial case easy calculus shows  that  any operator $T$ satisfying  the link condition \eqref{bm:link}, that is $\op\in\cTdw$,  acts like integrating   $(a)$-times and hence it is called {\it finitely smoothing} (c.f. \cite{Natterer84}).  On the other hand in the exponential case  it can  easily be seen that  $\op\in\cT_\tw^\td$ implies $\cR(\op)\subset \cW_{r}$ for all $r>0$, therefore  the operator $\op$ is called {\it infinitely smoothing} (c.f. \cite{Mair94}).  It is worth noting that these  are the usually studied cases in the literature (c.f. \cite{HallHorowitz2005}, \cite{ChenReiss2008} or \cite{JoVBVa07} in the context of nonparametric estimation of the structural function itself). However, the general results in the last section can be also applied considering more sophisticated sequences. Nevertheless, since in both cases the minimal regularity conditions given in Assumption \ref{ass:reg} are satisfied, the lower bounds presented in the next assertion follow directly from Theorem \ref{res:lower}.  

\begin{theo}\label{res:lower:sob}\dr Under the assumptions of Theorem \ref{res:lower} we have for any estimator $ \breve{\ell}$\\[-4ex]
\begin{itemize}\item[(i)] in the polynomial case, i.e. $\tw_1=1$ and $\tw_j  = |j|^{-2a}$, $j\geqslant 2$, for some $a>0$,    that\\[1ex] 
\hspace*{5ex}$\sup_{U\in\cU_\sigma}\sup_{\sol \in \cW_p^\br}\sup_{\rep \in \cW_s^\hr}  \bigl\{ \Ex|\breve{\ell}-\ell_\rep(\sol)|^2\bigr\}\gtrsim \max(n^{-(p+s)/(p+a)},n^{-1}) $,
\item[(ii)] in the exponential case, i.e. $\tw_1=1$ and $\tw_j  = \exp(-|j|^{2a})$, $j\geqslant 2$, for some $a>0$,    that\\[1ex] 
\hspace*{5ex}$\sup_{U\in\cU_\sigma}\sup_{\sol \in \cW_p^\br}\sup_{\rep \in \cW_s^\hr}  \bigl\{ \Ex|\breve{\ell}-\ell_\rep(\sol)|^2\bigr\}\gtrsim (\log n)^{-(p+s)/a}$.
\end{itemize}
\end{theo}

Let us  introduce now the second basis $\{f_l\}_{l\geq1}$, which is in  this section also given by the trigonometric basis. In this situation  the additional  moment conditions formalized in Assumption \ref{ass:A1}-\ref{ass:A3} are automatically fulfilled since both basis $\{\basZ_j\}_{j\geq1}$ and $\{f_l\}_{l\geq1}$ are uniformly bounded. However, we suppose that the associated conditional expectation operator $T$ satisfies the extended link condition \eqref{bm:link:gen}, that is, $T\in\cTdDw$. Thereby, we restrict the set of possible joint distributions of $(Z,W)$ to those having the trigonometric basis as optimal instruments. On the other hand, if the dimension  $m$ and the threshold $\alpha$ in the definition of  the  estimator $\widehat{\ell}_h$ given in \eqref{gen:def:est} are chosen appropriate, then by applying Theorem \ref{res:upper:A3}  the rates of the lower bound given in the last assertion provide up to a constant also the upper bound of the risk of $\widehat{\ell}_h$, which is summarized in the next theorem. 
We have thus proved that these rates are  optimal and the proposed estimator $\widehat{\ell}_h$ is minimax-optimal in both cases.

\begin{theo}\label{res:upper:sob}\dr  Assume an $n$-sample of $(Y,Z,W)$ from the model (\ref{model:NP}--\ref{model:NP2}) with  error term $U\in\cU_\sigma$, $\sigma>0$, and associated conditional expectation operator $T\in\cTdDw$, $\td,\tD\geq1$. Consider the estimator  ${\widehat\ell_\rep}$ given in \eqref{gen:def:est}\\[-4ex]
\begin{itemize}\item[(i)] in the polynomial case, i.e. $\tw_1=1$ and $\tw_j  = |j|^{-2a}$, $j\geqslant 2$, for some $a>0$,  with $m\sim n^{1/(2p+2a)}$ and threshold $\alpha\sim n$. If in addition  $p \geq 3/2$ then\\[1ex] 
\hspace*{5ex}$\sup_{\sol \in \cW_{p}^\br, \rep\in\cW_{s}^\hr} \{ \Ex|\widehat\ell_\rep-\ell_\rep(\sol)|^2\}\lesssim \max(n^{-(p+s)/(p+a)},n^{-1})$,
\item[(ii)] in the exponential case, i.e. $\tw_1=1$ and $\tw_j  = \exp(-|j|^{2a})$, $j\geqslant 2$, for some $a>0$,  with $m\sim (\log n)^{1/(2a)}$ and threshold $\alpha\sim n$. If in addition  $p \geq 3/2$, then\\[1ex] 
\hspace*{5ex}$\sup_{\sol \in \cW_{p}^\br, \rep\in\cW_{s}^\hr} \{ \Ex|\widehat\ell_\rep-\ell_\rep(\sol)|^2\}\lesssim  (\log n)^{-(p+s)/a}$.
\end{itemize}
 \end{theo}

\begin{rem}\label{rem:upper:sob:1}We shall emphasize the interesting influence of the parameters $p$, $s$ and $a$ characterizing the smoothness of $\sol$, $\rep$ and the smoothing properties  of  $\op$ respectively. As we see from Theorem \ref{res:lower:sob}  and \ref{res:upper:sob}, if the value of $a$ increases the obtainable optimal rate of convergence decreases. Therefore, the parameter $a$ is often called {\it degree of ill-posedness} (c.f. \cite{Natterer84}).  On the other hand, an increasing of the value $p+s$ leads to a faster optimal rate. In other words, as expected, values of a linear functional given by a smoother structural function or representer can be estimated faster. Moreover, in the polynomial case independent of the imposed smoothness assumption on the slope parameter (only $p\geq 3/2$ is needed) the parametric rate $n^{-1}$ is obtained if and only if the representer  is smoother than the degree of ill-posedness of $\op$, i.e., $s\geqslant a$. The situation is different  in the exponential case.   As long as the representer $\rep$ is only finitely times differentiable,   then due to  Theorem \ref{res:lower:sob}  and \ref{res:upper:sob} the optimal rate of convergence is logarithmic. However, if we restrict the class of representers even more, e.g. by considering $\cF_\hw^\hr$ with weights $\hw_1:=1$, $\hw_j=\exp(|j|^{2q}),j\geq 2$, which contains only analytic functions given $q>1$ (c.f. \cite{Kawata1972}).  Then faster rates are possible. Again independent of the imposed smoothness assumption on the structural parameter (again $p\geq 3/2$ is needed)  the parametric rate $n^{-1}$ is obtained if and only if the representer $\rep$ is smoother than the degree of ill-posedness of $\op$, e.g., $q\geqslant a$. Finally,  in opposite to the polynomial case in the exponential case  the smoothing parameter $m$ does not depend  on the value of $p$. It follows that the proposed estimator is automatically adaptive, i.e., it does not  depend on  an a-priori knowledge of the degree of smoothness of the structural function $\sol$. However, the choice of the smoothing parameter depends on  the smoothing properties  of $\op$, i.e., the value of $a$.\hfill$\square$\end{rem}
\appendix
\section{Appendix}\label{app:proofs}
\subsection{Proofs of Section \ref{sec:lower}.}
Consider the conditional expectation operator $\op$ associated to the regressor $Z$ and the instrument $W$, then $\Ex |[T\basZ_j](W)|^2 = \normV{\op\basZ_j}^2_W$, $j\in\N$. Therefore, if the link condition  \eqref{bm:link}, that is  $\op\in\cTdw$, is satisfied, then it follows that  $\Ex |[T\basZ_j](W)|^2 \asymp_{{d}} \tw_j$, for all $j\in\N$. This result will be used below without further reference. We shall prove at the end of this section the technical Lemma  \ref{app:lower:l1} used in the next proof.

\paragraph{Proof of the lower bound.}
\begin{proof}[\textcolor{darkred}{\sc Proof of Theorem \ref{res:lower}.}] We show below  for any estimator $\breve{\ell}$ 
only based on an $n$-sample of $(Y,Z,W)$ from the model  (\ref{model:NP}--\ref{model:NP2}) the following  two lower bounds: 
\begin{gather}\label{pr:lower:e1}
\sup_{U\in\cU_\sigma}\sup_{\sol \in \cF_\bw^{\br}}\sup_{\rep\in \cF_\hw^\hr} \Ex|\breve{\ell}-\ell_{\rep}(\sol)|^2 \geqslant \dstar \,\frac{\hr}{4\,{\triangle}}\,\min \Bigl( \frac{1}{2\, \td}\,,\,  \frac{\br}{\triangle}\Bigr) ,\\\label{pr:lower:e2}
\sup_{U\in\cU_\sigma}\sup_{\sol \in \cF_\bw^{\br}}\sup_{\rep\in \cF_\hw^\hr} \Ex|\breve{\ell}-\ell_{\rep}(\sol)|^2 \geqslant  \frac{1}{n}\, \frac{\hr}{4}\,\min \Bigl( \frac{1}{2\, \td}\,,\,  {\br} \Bigr) .
\end{gather}
Consequently, the result  follows  by combination of these two lower bounds.

 Proof of \eqref{pr:lower:e1}. Consider $(Z,W)$ with  associated  $T\in \cTdw$. Define the structural function $\sol_*:= [\sol_*]_{\kstar}\basZ_{\kstar}$, where $\kstar$ satisfies \eqref{res:lower:def:md} for some $\triangle\geqslant 1$ and  $[\sol_*]_{\kstar}$ is given in \eqref{app:l1:lower:bj} (Lemma \ref{app:lower:l1}). Then from \eqref{app:l1:lower:e2} in Lemma \ref{app:lower:l1}  follows  $\sol_* \in\cF_\bw^\br$ and thus $\sol_*^{(\theta)}:=\theta \sol_*\in\cF_\bw^\br$ with  $\theta\in\{-1,1\}$. Let $V$ be a Gaussian random variable with mean zero and variance one ($V \sim \cN(0,1)$)  which is   independent of  $(Z,W)$. Then  
$U_\theta:= [T\sol_*^{(\theta)}](W)-\sol_*^{(\theta)}(Z) +V $ belongs to $\cU_\sigma$ for all sufficiently large $\sigma$, since $\Ex U_\theta|W=0$ and   $\Ex [U_\theta^4|W]\leq 
8\{ 16 \br^2 \eta +3\}$. Consequently, for each $\theta$  i.i.d. copies $(Y_i,Z_i,W_i)$, $1\leq i\leq n$, of  $(Y,Z,W)$ with $Y:=\sol^{\theta}_*(Z)+U_\theta$  form an $n$-sample of the model (\ref{model:NP}--\ref{model:NP2}) and we denote their joint distribution  by  $P_{\theta}$.  In case of  $P_\theta$ the conditional distribution of $Y_i$ given $W_i$  is then Gaussian with mean  $ \theta [T\sol_*](W_i)$  and variance $1$. Then, it is easily seen that   the log-likelihood of ${P}_{1}$ with respect to  ${P}_{-1}$ is given by 
\begin{equation*}
\log\Bigl(\frac{d{P}_{1}}{d{P}_{-1}}\Bigr)=\sum_{i=1}^n 2(Y_i - [T\sol_*](W_i)) [T\sol_*](W_i) + \sum_{i=1}^n  2|[T\sol_*](W_i)|^2.
\end{equation*}
Its expectation with respect to ${P}_{1}$ satisfies 
$\Ex_{{P}_{1}}[\log(d{P}_{1}/d{P}_{-1})]=  2n \normV{T\sol_*}^2 \leq 2n d [\sol_*]_{\kstar}^2 \tw_{\kstar}$  by using $T\in\cTdw$. In terms of  Kullback-Leibler divergence this means  $KL(P_{1},P_{-1})\leqslant  2\,d\,n\,   [\sol_*]_{\kstar}^2 \tw_{\kstar}$. Since the
 Hellinger distance $H(P_{1},P_{-1})$ satisfies
 $H^2(P_{1},P_{-1}) \leqslant KL(P_{1},P_{-1})$  it follows 
from  \eqref{app:l1:lower:e2} in Lemma \ref{app:lower:l1} that 
\begin{equation}\label{pr:lower:e3}
H^2(P_{1},P_{-1}) \leqslant 2\,d\,n \, [\sol_*]_{\kstar}^2 \tw_{\kstar}\leqslant 1.
\end{equation} 
Consider the  Hellinger affinity $\rho(P_{1},P_{-1})= \int \sqrt{dP_{1}dP_{-1}}$ then we obtain for any estimator $\breve{\ell}$ and for all  $h\in \cF_\hw^\hr$ that  
\begin{align}\nonumber
\rho(P_{1},P_{-1})&\leqslant 
\int \frac{
|
\breve{\ell}
-
\ell_{h}(\sol_*^{(1)})
|
}{2
|\ell_{h}(\sol_*)|
} 
\sqrt{dP_1dP_{-1}} 
+
\int \frac{
|
\breve{\ell}
-
\ell_{h}(\sol_*^{(-1)})
|
}{2
|\ell_{h}(\sol_*)|
} 
\sqrt{dP_{1}dP_{-1}} 
\\\label{pr:lower:e4}
&\leqslant 
\Bigl( 
\int  \frac{|
\breve{\ell}
-
\ell_{h}(\sol_*^{(1)})
|^2}
{4
|\ell_{h}(\sol_*)|
^2}dP_{1}
\Bigr)^{1/2}
+
\Bigl( 
\int  \frac{|
\breve{\ell}
-
\ell_{h}(\sol_*^{(-1)})
|^2}
{4
|\ell_{h}(\sol_*)|^2}dP_{-1}
\Bigr)^{1/2}.
\end{align}
Due to the identity $\rho(P_{1},P_{-1})=1-\frac{1}{2}H^2(P_{1},P_{-1})$   combining  \eqref{pr:lower:e3} with 
 \eqref{pr:lower:e4} yields
\begin{equation}\label{pr:lower:e5}
\Bigl\{\Ex_{{P_{1}}}|
\breve{\ell}
-
\ell_{h}(\sol_*^{(1)})
|^2
+ 
\Ex_{{P_{-1}}}
|\breve{\ell}
-
\ell_{h}(\sol_*^{(-1)})
|^2
\Bigr\}\geqslant
\frac{1}{2}
|\ell_{h}(\sol_*)|
^2.
 \end{equation}
Consider now the representer  $\rep_*:=[\rep_*]_{\kstar}e_{\kstar}$, where  $[\rep_*]_{\kstar}^2:=\hr/\hw_{\kstar}$. Then by construction $\rep_*\in \cF_\hw^\hr$ and $|\ell_{\rep_*}(\sol_*)|^2=[\rep_*]_{\kstar}^2 [\sol_*]_{\kstar}^2\geqslant ({\hr}/{\triangle})\,\min (1/(2 \td),  \br/\triangle) \,\dstar$ by using  \eqref{app:l1:lower:e2} in Lemma \ref{app:lower:l1}. From \eqref{pr:lower:e5} together with the last estimate   we conclude  that
\begin{align*}
\sup_{U\in U_\sigma} \sup_{\sol\in \cF_\bw^\br}\sup_{\rep\in \cF_\hw^\hr} &\Ex|\breve{\ell}-\ell_{\rep}(\sol)|^2 \geqslant \sup_{\theta\in \{-1,1\}} \Ex_{P_\theta}|\breve{\ell} -\ell_{\rep_*}(\sol_*^{(\theta)})|^2\\
&\geqslant \frac{1}{2}
\Bigl\{\Ex_{{P_{1}}}|
\breve{\ell}
-
\ell_{\rep_*}(\sol_*^{(1)})
|^2
+ 
\Ex_{{P_{-1}}}
|\breve{\ell}
-\ell_{\rep_*}(\sol_*^{(-1)})
|^2
\Bigr\}\\
&\geqslant ({1}/{4})\,
[\rep_*]_{\kstar}^2 [\sol_*]_{\kstar}^2\geqslant (\dstar/{4}) ({\hr}/{\triangle})\,\min ( 1/(2 \td),  \br/\triangle),
\end{align*}
which proves \eqref{pr:lower:e1}. The proof of \eqref{pr:lower:e2} is similar to the proof of \eqref{pr:lower:e1}, but uses \eqref{app:l1:lower:e1}  in Lemma \ref{app:lower:l1} rather than \eqref{app:l1:lower:e2}. To be more precise, we  define the   structural function $\sol_*:= [\sol_*]_{1}\,e_{1}$, and the representer $[h_*]:=[h_*]_1e_1$, where $[\sol_*]_{1}$  and $[h_*]_1$ are given in \eqref{app:l1:lower:b1} (Lemma \ref{app:lower:l1}). Then by following along the same lines as in the proof of \eqref{pr:lower:e1} we obtain \eqref{pr:lower:e2}, which completes the proof. \end{proof}

\begin{lem}\label{app:lower:l1}Consider sequences $\tw$, $\bw$ and $\hw$  satisfying Assumption \ref{ass:reg}. Let $\kstar$ and $\dstar$ be such that \eqref{res:lower:def:md} holds true for some $\triangle\geq 1$.  If we define 
\begin{gather}\label{app:l1:lower:b1}[\rep_*]_1^2:= \hr, \qquad [\sol_*]_1^2:= \frac{\xi_1}{n }, \quad \text{ with }\quad \xi_1:=\min \left\{ \frac{1}{2\td}, {\br} \right\},\\\label{app:l1:lower:bj}[\rep_*]_{\kstar}^2:= \frac{\hr}{\hw_{\kstar}} \quad \mbox{ and }\quad [\sol_*]_{\kstar}^2:= \frac{\xi}{n\cdot \tw_{\kstar}}, \quad \text{ where }\quad \xi:=\min \left\{ \frac{1}{2 \td},  \frac{\rho}{\triangle} \right\}.\end{gather}
 Then we have
\begin{gather}\label{app:l1:lower:e1}
2dn\tw_1 [\sol_*]_{1}^2  \leqslant 1;\; 
\bw_1[\sol_*]^2_1 \leqslant \rho;\;
 [\rep_*]_1^2 \,[\sol_*]_1^2 \geqslant ({1}/{n})\,{\hr}\,\min ({1}/{(2\td)}, \br );\\\label{app:l1:lower:e2}
{2dn\tw_{\kstar}} [\sol_*]_{\kstar}^2  \leqslant 1;\;
\bw_{\kstar}[\sol_*]^2_{\kstar} \leqslant \rho;\;
 [\rep_*]_{m^*}^2 [\sol_*]_{m^*}^2 \geqslant \dstar\,({\hr}/{\triangle})\min( {1}/({2 \td}),  {\br}/{ \triangle}) .
\end{gather}
\end{lem}
\begin{proof}[\textcolor{darkred}{\sc Proof.}] We only  prove \eqref{app:l1:lower:e2}. The proof of \eqref{app:l1:lower:e1}  follows analogously and we omit the details. The first inequality in  \eqref{app:l1:lower:e2} is obtained trivially by using the definition of $\xi$. The second and third inequality in \eqref{app:l1:lower:e2} follows from 
 the definition of $\kstar$ and  $\dstar$ given in \eqref{res:lower:def:md}, i.e.,  $\bw_{\kstar} [\sol_*]_{\kstar}^2  \leq \xi\,\triangle$ and  $[\rep_*]^2_{\kstar} [\sol_*]_{\kstar}^2 = \xi\, \hr   \,  ({\bw_{\kstar}}/{(n\tw_{\kstar})})\,\dstar \geqslant \xi (\hr/{\triangle }) \,\dstar$, together with the definition of $\xi$, which completes the proof.\end{proof}

\subsection{Proofs of Section \ref{sec:gen}.}\label{app:proofs:gen}
We begin by defining and recalling notations to be used in the proofs of this section. Given $m>0$, denote $\sol_m:=\sum_{j=1}^m[\sol_m]_{j}\basZ_j$ with $[\sol_m]_{\um}=[\op]_{\um}^{-1}[g]_{\um}$ which is well-defined since $[T]_{\um}$ is  non singular. Then, the identities  $[T(\sol-\sol_m)]_{\um}=0$ and $[\sol_m-E_m \sol]_{\um} = [T]_{\um}^{-1}[TE_m^\perp \sol]_{\um}$ hold true.   Furthermore, let $[\Xi]_{\um}:= [\widehat T]_{\um}- [T]_{\um}$  and define vector $[B]_{\um}$ and $[S]_{\um}$ by 
\begin{equation}\label{app:gen:l:def}
[B]_j:=\frac{1}{n}\sum_{i=1}^n U_i f_j(W_i),\; [S]_j:=\frac{1}{n}\sum_{i=1}^n f_j(W_i)\{ \sol(Z_i) -   [\sol_m]_{\um}^t[e]_{\um}(Z_i)\},\;1\leq j\leq m,
\end{equation} 
where $[\widehat{g}]_{\um}- [\widehat T]_{\um} [\sol_m]_{\um}=[B]_{\um}+ [S]_{\um}$. Note that $ \Ex [B]_{\um}=0$ due to the mean independence, i.e., $\Ex(U|W)=0$, and that $\Ex[S]_{\um}= [T\sol]_{\um} - [T\sol_m]_{\um}= 0$. Moreover, let us introduce the events 
 \begin{multline}\label{app:l:upp:def:o}
\Omega:=\{ \normV{[\widehat{T}]^{-1}_{\um}}\leq \alpha\},\quad  \Omega_{1/2}:= \{\normV{[\Xi]_{\um}}\normV{[T]_{\um}^{-1}}\leq 1/2\}\\
\Omega^c:=\{ \normV{[\widehat{T}]^{-1}_{\um}}> \alpha\}\quad\mbox{ and }\quad  \Omega_{1/2}^c=\{\normV{[\Xi]_{\um}}\normV{[T]_{\um}^{-1}}> 1/2\}.
\end{multline}
Observe that $\Omega_{1/2} \subset\Omega$ in case $\alpha \geqslant2\normV{[\op]^{-1}_{\um}}$. Indeed, if $\normV{[\Xi]_{\um}}\normV{[T]_{\um}^{-1}}\leqslant 1/2$  then the identity $[\hop]_{\um}= [\op]_{\um}\{I+[\op]^{-1}_{\um}[\Xi_{n}]_{\um}\}$ implies by the usual Neumann series argument that $\normV{[\hop]^{-1}_{\um}} \leqslant 2\normV{[\op]^{-1}_{\um}}$. Thereby, if $\alpha \geqslant 2\normV{[\op]^{-1}_{\um}}$, then we have $\Omega_{1/2} \subset\Omega$. These results will be used below without further reference.

We shall prove in the end of this section four technical Lemma (\ref{app:gen:upper:l1} -- \ref{app:gen:upper:l4}) which are used in the following proofs. 
\paragraph{Proof of the consistency.}
\begin{proof}[\textcolor{darkred}{\sc Proof of Proposition \ref{res:gen:prop:cons}.}] Let $\widehat\ell_\rep^{\alpha}:=\ell_\rep(\sol_m)\1\{\normV{[\widehat{T}]^{-1}_{\um}}\leqslant \alpha\}$. Then the proof is based on the  decomposition 
\begin{equation}\label{app:gen:dec}
\Ex|\widehat\ell_\rep - \ell_\rep(\sol)|^2\leqslant 2\{ \Ex|\widehat\ell_\rep- \widehat\ell_\rep^{\alpha}|^2 + \Ex|\widehat\ell_\rep^{\alpha} - \ell_\rep(\sol)|^2\}.
\end{equation}
Under the assumption $\alpha \geq2\normV{[{T}]^{-1}_{\um}}$  we show  below    that for all $n\geq 1$
\begin{align}\label{app:gen:prop:cons:e1}
\Ex|{\widehat\ell_\rep}-{\widehat\ell_\rep^\alpha}|^2&\leqslant  2\normV{\rep}^2 \cdot \frac{m\, \alpha^2}{n} \cdot (\eta \cdot\normV{ \sol -   \sol_m}^2 +\sigma^2),\\\label{app:gen:prop:cons:e2}
\Ex|{\widehat\ell_\rep^\alpha}-{\ell_\rep(\sol)}|^2&\leqslant 2\normV{\rep}^2\,\Bigl\{\normV{\sol-\sol_m}^2+ \normV{\sol_m}^2\cdot\eta\cdot\frac{m^2\,\alpha^2}{n}\Bigr\}.
\end{align}
Moreover, we have  $\normV{\sol-\sol_m}=o(1)$ as $m\to\infty$, which can be realized as follows. Consider the decomposition $\normV{\sol-\sol_m}\leq \normV{E_m^\perp  \sol} +\normV{E_m \sol -\sol_m} $, where $\normV{E_m^\perp  \sol}=o(1)$  by using Lebesgue's dominated convergence theorem. The consistency of $\sol_m$ follows then from $\normV{E_m \sol -\sol_m}\leqslant\normV{E_m^\perp \sol   } \sup_m \normV{T^{-1}_{m} F_m T E_m^\perp}= O(\normV{E_m^\perp \sol   })$. Consequently, the
conditions on  $m$ and $\alpha$  ensure the convergence to zero as $n\to\infty$ of  the  bound given  in \eqref{app:gen:prop:cons:e1} and \eqref{app:gen:prop:cons:e2},  respectively, which proves the result.

Proof of \eqref{app:gen:prop:cons:e1}. By making use of the 
 identity $[\widehat{g}]_{\um}-[\widehat{T}]_{\um}[\sol_m]_{\um}= [B]_{\um}+ [S]_{\um}$  the Cauchy-Schwarz inequality and $ \normV{[\widehat{T}]_{\um}^{-1}}\1_\Omega \leq \alpha$ imply together 
\begin{equation*}
\Ex|{\widehat\ell_\rep}-{\widehat\ell_\rep^\alpha}|^2\leqslant \normV{\rep}^2  \cdot \alpha^2 \cdot 
\Ex\normV{ [B]_{\um}+ [S]_{\um}}^{2}.
\end{equation*}
and hence  (\ref{app:gen:prop:cons:e1}) follows from  \eqref{app:gen:upper:l1:e1:1} and \eqref{app:gen:upper:l1:e1:2}   in Lemma \ref{app:gen:upper:l1}.

The estimate (\ref{app:gen:prop:cons:e2}) follows from the decomposition
\begin{equation*}\Ex|{\widehat\ell_\rep^\alpha}-{\ell_\rep(\sol)}|^2\leq 2\normV{\rep}^2\{\normV{\sol-\sol_m}^2+ \normV{\sol_m}^2P(\Omega^c)\},
\end{equation*}
 where  we claim that $ P(\Omega^c)\leqslant 4 \eta\, m^2\normV{[{T}]^{-1}_{\um}}^2/n\leqslant  \eta\, m^2\,\alpha^2/n$. Indeed, since $\alpha \geqslant 2\normV{[{T}]^{-1}_{\um}}$ it follows that  $\Omega^c \subset\Omega_{1/2}^c$ and thus by applying Markov's inequality   we obtain from \eqref{app:gen:upper:l1:e2}   in Lemma \ref{app:gen:upper:l1} the estimate, which completes the proof.\end{proof}

\begin{proof}[\textcolor{darkred}{\sc Proof of Corollary \ref{res:gen:coro:cons}.}] By combination of   the identity $[\sol_m-E_m \sol]_{\um} = [T]_{\um}^{-1}[TE_m^\perp \sol]_{\um}$ and the estimate \eqref{app:gen:upper:l3:e1:1} in the proof of Lemma \ref{app:gen:upper:l3} with $\bw\equiv 1$  the extended link condition \eqref{bm:link:gen}, that is $T\in\cTdDw$, implies $\normV{T^{-1}_{m} F_m T E_m^\perp}^2=\sup_{\normV{\sol}=1}\normV{E_m \sol-\sol_m}^2\leqslant \tD\,\td$. Moreover, $2\normV{[\op]^{-1}_{\um}}\leq  2\normV{[\Diag(\tw)]^{-1/2}_{\um}} \normV{[\Diag(\tw)]^{1/2}_{\um}[\op]^{-1}_{\um}}\leqslant 2\sqrt{D/\tw_m}=\alpha$ since $\tw$ is non increasing. By using these estimates the result follows directly from 
Proposition \ref{res:gen:prop:cons}.\end{proof}

\paragraph{Proof of the upper bound.}
\begin{proof}[\textcolor{darkred}{\sc Proof of Theorem \ref{res:upper}.}]Our proof starts with the observation that by using   the definition \eqref{res:lower:def:md} of $\kstar$, that is, $1/\tw_{\kstar}\leqslant n\triangle/\bw_{\kstar}$,  the condition on the dimension $m=\kstar$ implies  that $m^3/(n\tw_{m})\leqslant\triangle m^3/\bw_{m}$ and  that the threshold satisfies   both $\alpha^2 = n\max(1, 4\, \tD\,\triangle/  \bw_{m})\geqslant 4\normV{[\op]^{-1}_{m}}^2$ and $\alpha^2/n\leqslant 4\, \tD\,\triangle$. On the other hand,   
 we show  below under the condition $\alpha\geq 2\normV{[T]^{-1}_{\um}} $ the following two bounds:   
\begin{align}\nonumber
\Ex|\widehat\ell_\rep- \widehat\ell_\rep^\alpha|^2
&\leqslant   (C/n)\,\normV{[\Diag(\hw\tw)]_{\um}^{-1}}\,\normV{\rep}^2_\hw\, \tD\,\eta^4\, ( \sigma^2+\Gamma\, \normV{ \sol -   \sol_m}_\bw^2 )\\ \label{app:gen:upper:e1}
&\hspace*{30ex}\Bigl\{ 1+4\tD \frac{ m^3}{\tw_m n}  + \frac{\alpha^2 m^3}{n} |P(\Omega^c_{1/2})|^{1/4} \Bigr\},\\\label{app:gen:upper:e2}
\Ex|\widehat\ell_\rep^\alpha-\ell_\rep(\sol)|^2&\leqslant 2\Bigl\{\normV{\rep}^2\normV{\sol_m}^2P(\Omega_{1/2}^c) + \bw_m^{-1}\max(\hw_m^{-1},\tw_m) \,2\tD\,\td\,\hwtwD \normV{\sol}_\bw^2\,\normV{\rep}_\hw^2\Bigr\},
\end{align}
for some generic  constant $C>0$ uniformly for all $n\in\N$, where $\normV{\sol_m}^2\leqslant2\{ \normV{\sol-\sol_m}^2+\normV{\sol}^2\}\leqslant2\{2\tD\td+1\} \normV{\sol}_\bw^2 \leqslant 6\tD\td \normV{\sol}_\bw^2$  and $ \normV{\sol-\sol_m}_\bw^2\leqslant2\tD\td \normV{\sol}_\bw^2 $ due to \eqref{app:gen:upper:l3:e1}  in Lemma \ref{app:gen:upper:l3}. Thus, by using $\Omega_{1/2}^c\subset\{ \normV{[\Xi]_{\um}}^2> \tw_m/(4\tD)\}$ and    $\normV{[\Diag(\hw\tw)]_{\um}^{-1/2}}^2\leqslant \hwtwD\tw_m^{-1} \max(\hw_m^{-1},\tw_m)$ (Assumption \ref{ass:reg}) it follows again from the decomposition \eqref{app:gen:dec} by combination of \eqref{app:gen:upper:e1} and \eqref{app:gen:upper:e2} that uniformly for all $\sol\in \cF_\bw^\br$ and  $\rep\in \cF_\hw^\hr$
\begin{multline*}
\Ex|\widehat\ell_\rep- \ell_\rep(\sol)|^2\leqslant C \,\td\,\tD^2\,\hwtwD\,\eta^4\,\{\sigma^{2}+ \td\,\tD\,\Gamma \}\,\Bigl\{1+  \triangle\,\frac{ m^3}{\bw_m}  + \triangle\,m^3 |P(\Omega^c_{1/2})|^{1/4} \Bigr\}\,\hr\,\br \\\cdot\Bigl\{(n\tw_m)^{-1}  \max(\hw_m^{-1},\tw_m)
+\bw_m^{-1}\max(\hw_m^{-1},\tw_m)  +P(\normV{[\Xi]_{\um}}^2> \tw_m/(4\tD))\Bigr\} .
\end{multline*}
The result follows now from $\{(n\tw_{\kstar})^{-1} 
+\bw_{\kstar}^{-1}\}\max(\hw_{\kstar}^{-1},\tw_{\kstar}) \leqslant 2\triangle\, \max(\dstar,1/n)$ by using the definition of  $\dstar$ given in \eqref{res:lower:def:md}. 

Proof of \eqref{app:gen:upper:e1}. By making use of the 
 identity $[\widehat g]_{\um}-[\widehat T]_{\um}[\sol_m]_{\um}=[B]_{\um} + [S]_{\um}$  it follows%
\begin{align*}
\widehat\ell_\rep- \widehat\ell_\rep^\alpha&=[\rep]_{\um}^t\,\{[T]_{\um}^{-1}+[\op]_{\um}^{-1}([T]_{\um} - [\widehat T]_{\um})[\hop]_{\um}^{-1}\}\{[B]_{\um} + [S]_{\um}\}\1_\Omega \\
&=[\rep]_{\um}^t [T]_{\um}^{-1} \,\{[B]_{\um} + [S]_{\um}\}\1_{\Omega}-[\rep]_{\um}^t\,[\op]_{\um}^{-1}[\Xi]_{\um}[\hop]_{\um}^{-1}\,\{[B]_{\um} + [S]_{\um}\}\1_\Omega
\end{align*}
where  \eqref{app:gen:upper:l2:e1:1} and \eqref{app:gen:upper:l2:e1:2}  in Lemma \ref{app:gen:upper:l2} with $z^t:= [h]_{\um}^t [T]_{\um}^{-1} / \normV{[h]_{\um}^t [T]_{\um}^{-1}}$ imply together
\begin{gather}\label{app:gen:upper:e1:1}
\Ex|[\rep]_{\um}^t [T]_{\um}^{-1} \,\{[B]_{\um} +[S]_{\um}\}|^{2}\leqslant (2/n)\cdot \normV{[\rep]_{\um}^t [T]_{\um}^{-1}}^2\,\eta^2\, ( \sigma^2+\Gamma\, \normV{ \sol -   \sol_m}_\bw^2 ). \end{gather}
On the other hand  we show below that  there exists a generic constant $C>0$ such that
\begin{multline}\label{app:gen:upper:e1:2}
\Ex| [\rep]_{\um}^t\,[\widehat T]_{\um}^{-1}[\Xi]_{\um}[T]_{\um}^{-1}\,\{[B]_{\um}+[S]_{\um}\}\1_\Omega|^2
\leqslant (C/n)\,\normV{[\rep]_{\um}^t [T]_{\um}^{-1}}^2\,\eta^4\, ( \sigma^2+\Gamma\, \normV{ \sol -   \sol_m}_\bw^2 )\\ 
\Bigl\{ 4\tD \frac{ m^3}{\tw_m n}  + \frac{\alpha^2 m^3}{n} |P(\Omega^c_{1/2})|^{1/4} \Bigr\}  
\end{multline}
Consequently, the  inequality \eqref{app:gen:upper:e1} follows by combination of \eqref{app:gen:upper:e1:1}  and \eqref{app:gen:upper:e1:2} together with 
  $\normV{[\rep]_{\um}^t [T]_{\um}^{-1}}^2  \leqslant  \normV{[\rep]_{\um}^t[\Diag(\upsilon)]_{\um}^{-1/2}}^2 D \leqslant \normV{h}^2_\hw\,\normV{[\Diag(\hw\upsilon)]_{\um}^{-1/2}}^2D $ since  $T\in\cTdDw$. 

The proof of \eqref{app:gen:upper:e1:2}  starts with the observations that $T\in\cTdDw$ implies $\normV{[\hop]_m^{-1}}\1_{\Omega_{1/2}} \leq 2 \normV{[\op]_m^{-1}} \leq 2 \sqrt{D/\tw_m}$ and that  $\normV{[\hop]_m^{-1}}\1_\Omega \leq \alpha$.  By using these estimates  we obtain 
\begin{multline*}
\Ex|[\rep]_{\um}^t\,[\op]_{\um}^{-1} [\Xi]_{\um} [\hop]_{\um}^{-1}\{ [B]_{\um} +[S]_{\um}\}\1_\Omega|^2\\\leqslant 
\normV{[\rep]_{\um}^t\,[\op]_{\um}^{-1}}^2 \Bigl\{ 4 D\tw_m^{-1}  \Ex\normV{ [\Xi]_{\um}}^2\normV{[B]_{\um} +[S]_{\um}}^2\1_{\Omega_{1/2}} + 
\alpha^2 \Ex\normV{ [\Xi]_{\um}}^2\normV{[B]_{\um} +[S]_{\um}}^2\1_{\Omega_{1/2}^c}\Bigr\}\\
\leq \normV{[\rep]_{\um}^t\,[\op]_{\um}^{-1}}^2 \Bigl\{ 4 D\tw_m^{-1}  \bigl(\Ex \normV{[\Xi]_{\um}}^4\bigr)^{1/2} + \alpha^2 \bigl(\Ex \normV{[\Xi]_{\um}}^8\bigr)^{1/4} P(\Omega_{1/2}^c)^{1/4}\Bigr\}\bigl(\Ex\normV{[B]_{\um}+[S]_{\um}}^4\bigr)^{1/2}.
\end{multline*}
Consequently,  \eqref{app:gen:upper:l2:e2:1}, \eqref{app:gen:upper:l2:e2:2} and \eqref{app:gen:upper:l2:e3} in Lemma \ref{app:gen:upper:l2:e3} imply together \eqref{app:gen:upper:e1:2}. 

Proof of \eqref{app:gen:upper:e2}.   Following  along the lines of the proof of (\ref{app:gen:prop:cons:e2})  we obtain
\begin{equation*}\Ex|{\widehat\ell_\rep^\alpha-\ell_\rep(\sol)}|^2\leqslant 2\{|\skalarV{\rep,\sol-\sol_m}|^2+  \normV{\rep}^2\normV{\sol_m}^2P(\Omega^c_{1/2})\}.
\end{equation*}
Then, due to $\sol\in \cF_\bw^\br$ and  $\rep\in \cF_\hw^\hr$  the  estimate 
\eqref{app:gen:upper:l3:e2}   in Lemma \ref{app:gen:upper:l3} implies \eqref{app:gen:upper:e2}, which completes the proof. \end{proof}

\begin{proof}[\textcolor{darkred}{\sc Proof of Theorem \ref{res:upper:A3}.}]The result follows from Theorem \ref{res:upper} 
since  $\kstar^3\bw_{\kstar}^{-1} =O(1)$ by using the additional properties \eqref{gen:upper:varphi:cond} and 
\begin{align}
\label{app:gen:upper:A3:e1} &\kstar^3 \Bigl|P\Bigl(\normV{[\widehat T]_{\ukstar}-[T]_{\ukstar}}^2> \tw_{\kstar}/(4\tD)\Bigr)\Bigr|^{1/4}=O(1),\\
\label{app:gen:upper:A3:e2} &P\Bigl(\normV{[\widehat{T}]_{\ukstar}-[{T}]_{\ukstar}}^2> \frac{\tw_{\kstar}}{4\tD} \Bigr)=O(\max(\dstar,1/n)),\end{align} which can be realized as follows. Consider first \eqref{app:gen:upper:A3:e2}. From the  definition of $\kstar$ follows  $n\tw_{\kstar}\bw_{\kstar}^{-1}\geqslant \triangle^{-1}$. By using this estimate together with \eqref{app:gen:upper:l4:e1} in Lemma \ref{app:gen:upper:l4} we conclude
\begin{multline*}\kstar^3 \Bigl|P\Bigl(\normV{[\widehat T]_{\ukstar}-[T]_{\ukstar}}^2> \tw_{\kstar}/(4\tD)\Bigr)\Bigr|^{1/4}\\\hspace*{5ex}\leqslant 2^{1/4} \exp\{ -(n\tw_{\kstar}\kstar^{-2})/(80 \tD  \eta^2)+ (7/2)\log \kstar\}
\\\hspace*{5ex}\leqslant 2^{1/4} \exp\Bigl\{ -\frac{\bw_{\kstar}}{{\kstar^2}}\Bigl( \frac{1}{80 \tD  \eta^2 \triangle}- \frac{ \kstar^3}{\bw_{\kstar}} \frac{ (7/2) \log \kstar}{{\kstar}}\Bigr) \Bigr\}.\end{multline*}
Consequently, \eqref{app:gen:upper:A3:e1} follows also from the conditions \eqref{gen:upper:varphi:cond}, that is, $\kstar^3\bw_{\kstar}^{-1} =O(1)$.
Consider \eqref{app:gen:upper:A3:e2}. From the  definition of $\kstar$ moreover follows    $\min(\dstar^{-1},n)\leqslant \triangle\bw_{\kstar} \min(\tw_{\kstar}^{-1},\hw_{\kstar})$. This estimate and $n\tw_{\kstar}/\bw_{\kstar}\geqslant 1/\triangle$  together with \eqref{app:gen:upper:l4:e1} in Lemma \ref{app:gen:upper:l4} implies now
\begin{multline*}\min(\dstar^{-1},n) P(\normV{[\widehat{T}]_{\ukstar}-[{T}]_{\ukstar}}^2> \tw_{\kstar}/(4\tD)  )\\\hspace*{5ex}\leqslant 2 \exp\{ -(n\tw_{\kstar}\kstar^{-2})/(20 \tD  \eta^2)+ 2 \log \kstar +\log \min(\dstar^{-1},n)\}
\\\hspace*{5ex}\leqslant 2 \exp\Bigl\{ -\frac{\bw_{\kstar}}{{\kstar^2}}\Bigl( \frac{1}{20 \tD  \eta^2 \triangle}- \frac{ \kstar^3}{\bw_{\kstar}} \frac{ 2 \log \kstar+ \log \triangle}{{\kstar}}  \\
\hspace*{10ex} - \frac{\kstar^2\log \bw_{\kstar} }{\bw_{\kstar} } - \frac{ \kstar^2\log \min(\tw_{\kstar}^{-1},\hw_{\kstar})}{\bw_{\kstar}}\Bigr)\Bigl\}.\end{multline*}
Thus, the estimate \eqref{app:gen:upper:A3:e2} follows again by using the additional properties \eqref{gen:upper:varphi:cond}, 
which completes the proof.\end{proof}

\paragraph{Technical assertions.}\hfill\\[1ex]
The following paragraph gathers technical results used in the proof of Section  \ref{sec:gen}.
\begin{lem}\label{app:gen:upper:l1} Suppose that the error term $U$ satisfies $\Ex[ U^2|W]\leqslant\sigma^2$, $\sigma>0$ and that the joint distribution of  $(Z,W)$ fulfills Assumption \ref{ass:A1}. Then  for all $m\in\N$ we have 
\begin{gather}\label{app:gen:upper:l1:e1:1}
\Ex\normV{[B]_{\um}}^{2}\leqslant (m/n)\cdot\sigma^{2},\\\label{app:gen:upper:l1:e1:2}
\Ex\normV{[S]_{\um}}^{2}\leqslant  (m/n)\cdot   \eta \cdot  \normV{\sol-\sol_m}^2,
\\\label{app:gen:upper:l1:e2}
 \Ex\normV{[\Xi]_{\um}}^2\leqslant  (m^2/n) \cdot \eta  .
\end{gather}
\end{lem}
\begin{proof}[\textcolor{darkred}{\sc Proof.}] Proof of \eqref{app:gen:upper:l1:e1:1} and \eqref{app:gen:upper:l1:e1:2}. Consider $\Ex\normV{[B]_{\um}}^{2}=\sum_{j=1}^m\Ex|(1/n)\sum_{i=1}^n U_if_j(W_i)|^2$. By using the mean independence (Assumption \ref{ass:A1}), i.e.,  $ \Ex [U|W]= 0$, it follows that the random variables $(U_i f_{j}(W_i))$,  $1\leqslant i\leqslant n,$ are i.i.d. with mean zero, thus $\Ex\normV{[B]_{\um}}^{2}= (m/n) \Ex |Uf_j(W)|^2$. Thus, $\Ex (U^2|W)\leqslant \sigma^2$ and  $\Ex f_j^2(W)=1$ imply \eqref{app:gen:upper:l1:e1:1}. Consider  \eqref{app:gen:upper:l1:e1:2}, where  for each $1\leqslant j \leqslant m$ the random variables  $(\{\sol(Z_i) -   [\sol_m]_{\um}^t[e]_{\um}(Z_i)\}f_j(W_i))$,  $1\leqslant i\leqslant n,$ are i.i.d. with mean zero. Thus $\Ex\normV{[S]_{\um}}^{2}\leqslant (m/n) \sup_{j}\Ex |\{ \sol(Z) -   [\sol_m]_{\um}^t[e]_{\um}(Z)\}f_{j}(W)|^2$ and, hence \eqref{app:gen:upper:l1:e1:2} follows from Assumption \ref{ass:A1}, i.e., $\sup_{j\in\N}  \Ex [f_j^2(W)|Z]\leqslant\eta$, together with $\Ex \{ \sol(Z) -   [\sol_m]_{\um}^t[e]_{\um}(Z)\}^2= \normV{\sol-\sol_m}^2 $.

Proof of \eqref{app:gen:upper:l1:e2}. Let $1\leqslant j,l \leqslant m$. Then $(e_{j}(Z_i)f_l(W_i)-[T]_{j,l})$,  $1\leqslant i\leqslant n,$ are i.i.d. with mean zero and   $\Ex [\Xi]_{j,l}^{2}=  n^{-1} \Ex \{e_{j}(Z)f_l(W)-[T]_{j,l}\}^{2}\leqslant n^{-1} \eta $ by Assumption \ref{ass:A1}.  Consequently, \eqref{app:gen:upper:l1:e2}  follows from the   estimate $\Ex\normV{ [\Xi]_{\um}}^{2}\leq  \sum_{j,l=1}^m \Ex [\Xi]_{j,l}^{2}$, which completes the proof.\end{proof}

\begin{lem}\label{app:gen:upper:l2} Let $\mmS^m:=\{z\in\R^m:z^tz=1\}$.  Suppose that $U\in \cU_\sigma$, $\sigma>0$ and that the joint distribution of  $(Z,W)$ satisfies Assumption \ref{ass:A2}.  If in addition $\sol\in \cF_\bw^\br$ with $\Gamma= \sum_{j=1}^\infty \bw_j^{-1}<\infty $, then  there exists a  constant $C>0$ such that for all $m\in\N$
\begin{gather}\label{app:gen:upper:l2:e1:1}
\sup_{z\in \mmS^m}\{ \Ex|z^t \,[B]_{\um} |^{2}\} \leqslant (1/n)\, \sigma^2,\\\label{app:gen:upper:l2:e1:2}
\sup_{z\in \mmS^m}\{\Ex|z^t \,[S]_{\um}|^{2}\}\leqslant (1/n)\,\eta^2\,\Gamma\, \normV{ \sol -   \sol_m}_\bw^2 \\\label{app:gen:upper:l2:e2:1}
\Ex\normV{[B]_{\um}}^{4}\leqslant C\cdot \Bigl((m/n)\cdot\sigma^{2} \cdot \eta^2\Bigr)^2,\\\label{app:gen:upper:l2:e2:2}
\Ex\normV{[S]_{\um}}^{4}\leqslant C\cdot \Bigl((m/n)\cdot\eta^2 \cdot \Gamma\cdot  \normV{\sol-\sol_m}_\bw^2\Bigr)^2,
\\\label{app:gen:upper:l2:e3}
 \Ex\normV{[\Xi]_{\um}}^8\leqslant C \cdot \Bigl((m^2/n )\cdot  \eta^{2}\Bigr)^4.
\end{gather}
\end{lem}
\begin{proof}[\textcolor{darkred}{\sc Proof.}] Consider \eqref{app:gen:upper:l2:e1:1}. Let $z\in\mmS^m$. By using the mean independence, i.e.,  $ \Ex [U|W]= 0$, it follows that the random variables $(U_i \sum_{j=1}^m z_j f_{j}(W_i))$,  $1\leqslant i\leqslant n,$ are i.i.d. with mean zero. Therefore, we have $\Ex|z^t \,[B]_{\um} |^{2}=(1/n)\Ex |U \sum_{j=1}^m z_j f_{j}(W)|^2$. Then  \eqref{app:gen:upper:l2:e1:1} follows   from  $\Ex (U^2|W)\leqslant (\Ex (U^4|W))^{1/2}\leqslant \sigma^2$  and $\Ex[ f_j(W)f_l(W)]=\delta_{jl}$ with $\delta_{jl}= 1$ if $j=l$ and zero otherwise. Consider \eqref{app:gen:upper:l2:e1:2}. Since $ (f_j(W)\{ \sol(Z) -   [\sol_m]_{\um}^t[e]_{\um}(Z)\})$ has mean zero, it follows that  $(\{ \sol(Z_i) -   [\sol_m]_{\um}^t[e]_{\um}(Z_i)\} \sum_{j=1}^m z_j f_{j}(W_i))$,  $1\leqslant i\leqslant n,$ are i.i.d. with mean zero. Thus,  $\Ex|z^t \,[S]_{\um} |^{2}=(1/n)\Ex |\{ \sol(Z) -   [\sol_m]_{\um}^t[e]_{\um}(Z)\} \sum_{j=1}^m z_j f_{j}(W)|^2$. Then  \eqref{app:gen:upper:l2:e1:2} follows  from Assumption \ref{ass:A2} (i), i.e., $\sup_{l\in\N}\Ex [|e_l(Z)|^2|W] \leqslant \eta^2$, and $\Ex[ f_j(W)f_l(W)]=\delta_{jl}$. Indeed, by using the Cauchy-Schwarz inequality and that $\sum_{j\in\N}\bw_j^{-1}=\Gamma<\infty$ we have%
\begin{multline*}\Ex |\{ \sol(Z) -   [\sol_m]_{\um}^t[e]_{\um}(Z)\} \sum_{j=1}^m z_j f_{j}(W)|^2\leqslant \normV{\sol-\sol_m}_\bw^2 \sum_{l\in\N}\bw_l^{-1} \Ex |e_l(Z)\sum_{j=1}^m z_j f_{j}(W)|^2\\\leqslant  \normV{\sol-\sol_m}_\bw^2 \,\eta^2\,\sum_{l\in\N}\bw_l^{-1}\, \sum_{j=1}^m z_j^2 = \normV{\sol-\sol_m}_\bw^2 \,\eta^2\,\Gamma .\end{multline*}

Proof of  \eqref{app:gen:upper:l2:e2:1}. Since $\Ex\normV{[B]_{\um}}^{4}\leqslant m \sum_{j=1}^m\Ex|(1/n)\sum_{i=1}^n U_if_j(W_i)|^4$, where for each $1\leqslant j \leqslant m$ the random variables $(U_i  f_{j}(W_i))$,  $1\leqslant i\leqslant n,$ are i.i.d. with mean zero. It follow from  Theorem 2.10 in \cite{Petrov1995} that $\Ex|(1/n)\sum_{i=1}^n U_if_j(W_i)|^4\leqslant C n^{-2} \Ex | U f_j(W)|^4$ for some generic constant $C>0$. Thus, by using  $\Ex (U^4|W)\leqslant \sigma^4$ and $\sup_{j\in\N}  \Ex |f_j(W)|^4\leq \eta^4$ (Assumption \ref{ass:A2} (i)), we obtain \eqref{app:gen:upper:l2:e2:1}. 
The proof of  \eqref{app:gen:upper:l2:e2:2} follows  in analogy to the proof of \eqref{app:gen:upper:l2:e2:1}.  Observe  that  for each $1\leqslant j \leqslant m$,  $(\{\sol(Z_i) -   [\sol_m]_{\um}^t[e]_{\um}(Z_i)\}f_j(W_i))$,  $1\leqslant i\leqslant n,$ are i.i.d. with mean zero and $\Ex |\{ \sol(Z_i) -   [\sol_m]_{\um}^t[e]_{\um}(Z_i)\}f_{j}(W_i)|^4\leqslant  \eta^4 \, \Gamma^2\,  \normV{\sol-\sol_m}_\bw^4$, which can be realized as follows. Since $[T(\sol-\sol_m)]_j=0$ it follows that 
$\{\sol(Z) -   [\sol_m]_{\um}^t[e]_{\um}(Z)\}f_j(W)= \sum_{l\in\N} [\sol-\sol_m]_l \{e_l(Z)f_j(W)-[T]_{j,l}\}$. Furthermore, by using Assumption \ref{ass:A2} (ii), i.e., $\sup_{j,l\in\N}  \Ex |e_l(Z)f_j(W)-[T]_{j,l}|^4\leq 2\eta^4$, the Cauchy-Schwarz inequality implies%
\begin{multline*}\Ex |\{ \sol(Z) -   [\sol_m]_{\um}^t[e]_{\um}(Z)\}f_{j}(W)|^4\leqslant 
\normV{\sol-\sol_m}_\bw^4 \Ex\Bigl|\sum_{l\in\N}\bw_l^{-1} |e_l(Z)f_{j}(W)-[T]_{j,l}|^2\Bigr|^2\\
\leqslant \normV{\sol-\sol_m}_\bw^4  \Gamma^2  2\eta^4.\end{multline*}

Proof of \eqref{app:gen:upper:l2:e3}.  The random variables $(e_{l}(Z_i)f_j(W_i)-[T]_{j,l})$,  $1\leqslant i\leqslant n,$ are i.i.d. with mean zero for each $1\leqslant j,l \leqslant m$. Hence,  Theorem 2.10 in \cite{Petrov1995} together with Assumption \ref{ass:A2} (ii), i.e., $\sup_{j,l\in\N}  \Ex |e_l(Z)f_j(W)-[T]_{j,l}|^8\leq 8!\eta^8$,  implies  $\sum_{j,l=1}^m\Ex [\Xi]_{j,l}^{8}\leq C m^2 n^{-4} \eta^8$.  Consequently, \eqref{app:gen:upper:l2:e3}  follows from the   estimate $\Ex\normV{ [\Xi]_{\um}}^{8}\leq m^{6} \sum_{j,l=1}^m \Ex [\Xi]_{j,l}^{8}$, which completes the proof.\end{proof}

\begin{lem}\label{app:gen:upper:l3}
 Let  $g=T\sol$ and denote  $\sol_m:=[T]_{\um}^{-1}[g]_{\um}$, $m\in\N$. If $T\in\cT^{\tw}_{\td,\tD}$ and  $\sol\in \cF_\bw^\br$, then for all $0\leq s\leq 1$ we obtain
\begin{gather} \label{app:gen:upper:l3:e1}
\sup_{m\in\N}\{\bw_m^{1-s}\,\normV{\sol-\sol_m}_{\bw^s}^2\}\leqslant 2 \,\tD \, \td\, \br.
\end{gather}
If in addition $\rep\in\cF_\hw^\hr$, then  under Assumption \ref{ass:reg} we have
\begin{gather} \label{app:gen:upper:l3:e2}
\sup_{m\in\N}\{\bw_m\min(\hw_m,\upsilon_m^{-1}) \,|\skalarV{\rep,\sol-\sol_m}|^2\}\leqslant 2\,\hwtwD\, \tD\, \td \, \br\,\hr.
\end{gather}
\end{lem}
\begin{proof}[\textcolor{darkred}{\sc Proof.}]  Consider the decomposition
\begin{equation*}
\normV{\sol-\sol_m}^2_{\bw^s}\leqslant 2\{   \normV{\sol-E_m \sol}^2_{\bw^s} +\normV{E_m \sol-\sol_m}^2_{\bw^s} \}.
\end{equation*}
Since $(\bw_j^{s-1})$ is monotonically decreasing it follows  that $ \normV{\sol-E_m\sol}^2_{\bw^s} \leqslant \bw_m^{s-1}\,\normV{\sol}^2_\bw$, while we show below that 
\begin{equation}\label{app:gen:upper:l3:e1:1}
\normV{E_m \sol-\sol_m}^2_{\bw^s}\leqslant \tD\,\td\, \bw_m^{s-1}\,\normV{\sol}^2_\bw.
\end{equation}
Consequently,  by combination of these two bounds the condition $\sol\in \cF_\bw^\br$, i.e., $\normV{\sol}^2_\bw\leqslant \br$, implies \eqref{app:gen:upper:l3:e1}. Consider \eqref{app:gen:upper:l3:e1:1}. Since $T\in\cT^{\tw}_{\td,\tD}$, i.e., 
$\sup_{m\in\N}\normV{[\Diag(\tw)]^{1/2}_{\um} [T]_{\um}^{-1}}^2\leqslant \tD$ and $ \normV{ Tf}^2 \leqslant \td \normV{ f}_\tw^2$ for all $f\in L^2_Z$,
the identity $[E_m \sol-\sol_m]_{\um} = -[T]_{\um}^{-1}[TE_m^\perp \sol]_{\um}$ implies
$\normV{E_m\sol-\sol_m}^2_\tw \leqslant \tD \normV{ TE_m^\perp \sol}^2$ and hence
\begin{equation}\label{app:gen:upper:l3:e1:2}
\normV{E_m\sol-\sol_m}^2_\tw \leqslant \tD\, \td\,   \bw_m^{-1}\tw_m \normV{ \sol}^2_\bw
\end{equation}
because $(\bw_j^{-1}\tw_j)$ is monotonically decreasing.  Furthermore, since $(\bw_j^{s}\tw_j^{-1})$ is monotonically increasing we have $ \normV{E_m\sol-\sol_m}^2_{\bw^s} \leqslant \bw_m^{s}\tw_m^{-1}\,\normV{E_m\sol-\sol_m}^2_\tw$. The inequality \eqref{app:gen:upper:l3:e1:1} follows now by combination of the last estimate and \eqref{app:gen:upper:l3:e1:2}.

Proof of \eqref{app:gen:upper:l3:e2}. By applying the Cauchy-Schwarz inequality we have
\begin{equation}\label{app:gen:upper:l3:e2:1}
|\skalarV{\rep,\sol-E_m\sol}|^2\leqslant \hw_m^{-1}\bw^{-1}_m \,\normV{\rep}^2_\hw \normV{\sol}_\bw^2 
\end{equation}
and by using \eqref{app:gen:upper:l3:e1:2} it follows
\begin{multline}\label{app:gen:upper:l3:e2:2}
|\skalarV{\rep,E_m\sol-\sol_m}|^2\leqslant \normV{\rep}^2_\hw \normV{[\Diag(\hw)]_{\um}^{-1/2}[\Diag(\tw)]_{\um}^{-1/2}}^2
\normV{ (E_m\sol-\sol_m)}^2_{\tw}\\
\leqslant  \normV{\rep}^2_\hw \,  \{\sup_{1\leqslant j\leqslant m} 1/(\hw_j\tw_j)\}\, \tD\,\td\,\bw_m^{-1}\tw_m\, \normV{\sol}_\bw^2.
\end{multline}
Since under Assumption \ref{ass:reg} there exist a  constant $\hwtwD$  such that for all $m\in\N$ holds $\upsilon_m\sup_{1\leqslant j\leqslant m}\{ 1/(\hw_j\upsilon_j)\}\leqslant \hwtwD \max(\hw_m^{-1},\upsilon_m)$ the assertion  \eqref{app:gen:upper:l3:e2} follows from \eqref{app:gen:upper:l3:e2:1} and \eqref{app:gen:upper:l3:e2:1}, which completes the proof.\end{proof}

\begin{lem}\label{app:gen:upper:l4} Suppose that the joint distribution of  $(Z,W)$ satisfies Assumption \ref{ass:A3}. If in addition the sequence $\tw$ fulfills Assumption \ref{ass:reg},  then for all $m\in\N$ we have
\begin{gather}\label{app:gen:upper:l4:e1}
P(\normV{[\Xi]_{\um}}^2> \tw_{m}/(4\tD) )\leqslant 2 \exp\{ -(n\tw_{m}/m^2)/(20 \tD  \eta^2)+ 2 \log m\}.
\end{gather}
\end{lem}
\begin{proof}[\textcolor{darkred}{\sc Proof.}]Our proof starts with the observation that for all $j,l\in\N$ the condition (iii) in Assumption \ref{ass:A3} implies for all $t>0$ \begin{gather*}
P(| e_j(Z)f_l(W)- \Ex [e_j(Z)f_l(W)]| \geqslant t )\leqslant 2 \exp\{ -t^2/(4 n \Var (e_j(Z)f_l(W)) + 2\eta t )\},\end{gather*} which is just Bernstein's inequality (a detailed discussion can be found, for example, in \cite{Bosq1998}). Therefore,  the condition $\sup_{j,l\in\N}\Var (e_j(Z)f_l(W))\leqslant \eta^2$  (Assumption \ref{ass:A2} (ii)) implies now for all $t>0$
\begin{gather}\label{app:gen:upper:l4:e1:1}
\sup_{j,l\in\N}P(| e_j(Z)f_l(W)- \Ex [e_j(Z)f_l(W)]| \geqslant t )\leqslant 2 \exp\{ -t^2/(4 n\eta^2 + 2\eta t )\}.\end{gather}
On the other hand, it is well-known that $m^{-1}\normV{[A]_{\um}}\leqslant \max_{1\leqslant j,l\leqslant m }|[A]_{j,l}|$ for any $m\times m$ matrix $[A]_{\um}$. Combining the last estimate and \eqref{app:gen:upper:l4:e1:1} we obtain for all $t>0$
\begin{multline*}
P(m^{-1}\normV{[\Xi]_{\um}}\geqslant t)\leqslant \sum_{j,l=1}^m P(| e_j(Z)f_l(W)- \Ex [e_j(Z)f_l(W)]|\geqslant nt)\\ \leqslant 2  \exp\{ -(nt^2)/(4 \eta^2 + 2\eta t )+ 2 \log m\}.
\end{multline*}
From the last estimate it follows now 
\begin{equation*} 
P(\normV{[\Xi]_{\um}}^2> \tw_{m}/(4\tD) )
\leqslant 2 \exp\{ -(n\tw_{m}/m^2)/(4\tD (4 \eta^2 + (\eta/\tD^{1/2}) (\tw_{m}^{1/2}/m))+ 2 \log m\},
\end{equation*}
which together with Assumption \ref{ass:reg}, that is, $\tw_{m}^{1/2}/m\leqslant \eta \tD^{1/2}$,  implies the result.\end{proof}
\subsection{Proofs of Section \ref{sec:sob}}\label{app:proofs:sob}
\paragraph{The lower bounds.}
\begin{proof}[\textcolor{darkred}{\sc Proof of Theorem \ref{res:lower:sob}.}] Observe that 
$\cW_p^\br= \cF_\bw^\br$ and $\cW_s^\hr= \cF_\hw^\hr$ with  weights $\bw=(\bw_j)_{j\geqslant1}$ and $\hw=(\hw_j)_{j\geqslant1}$ given by $\bw_1:=1,\bw_j:=|j|^{2p}$ and $\hw_1:=1,$ $\hw_j:=|j|^{2s},j\geqslant2$, respectively. Obviously,  the sequences $\bw$, $\hw$ and $\tw$  given in (i) by $\tw=1,\tw_j=|j|^{-2a}$ and (ii) by $\tw=1,\tw_j=\exp(-|j|^{2a})$, $j\geqslant 2$, satisfy Assumption \ref{ass:reg}.  Furthermore, in case (i) we have  $1/(\bw_{\kstar}\tw_{\kstar})=  \kstar^{2a+2p}$. It follows that  $\kstar$ and $\dstar$ given in \eqref{res:lower:def:md} of Theorem \ref{res:lower} satisfies $\kstar\sim n^{1/(2p+2a)}$ and $\dstar\sim n^{-(p+s)/(p+a)}$ respectively. On the other hand,  $1/(\bw_{\kstar}\tw_{\kstar})= \kstar^{2p}\exp(\kstar^{2a})$ implies  in case (ii)  that   $\kstar\sim (\log n)^{1/(2a)}$ and $\dstar\sim (\log n)^{(p-s)/a}$. Consequently, the lower bounds in Theorem \ref{res:lower:sob} follow by applying Theorem \ref{res:lower}.\end{proof}
\paragraph{The upper bounds.}
\begin{proof}[\noindent\textcolor{darkred}{\sc Proof of Theorem \ref{res:upper:sob}.}] Observe that in both cases the  condition \eqref{gen:upper:varphi:cond} is satisfied if $p\geq 3/2$. Since the condition on $m$ and $\alpha$ ensures in both cases  that  $m\sim \kstar$ and $\alpha\sim n$ (see proof of Theorem \ref{res:lower:sob})  the result follows from Theorem \ref{res:upper:A3}.\end{proof}

\bibliography{BiB-NIR-LinFun}
\end{document}